\documentclass[12pt]{amsart}
\usepackage{amsmath,amssymb,amsfonts}
\input xy
\xyoption{all}

\theoremstyle{plain}
\newtheorem{thm}{Theorem}[section]

\newtheorem{lem}[thm]{Lemma}

\newtheorem{cor}[thm]{Corollary}

\newtheorem{prop}[thm]{Proposition}

\newtheorem{conj}{Conjecture}

\theoremstyle{definition} \theoremstyle{definition}

\newtheorem{rem}[thm]{Remark}           

\theoremstyle{remark}
\newtheorem*{claim}{Claim}

\newcommand{\A}{\mathbb{A}}
\newcommand{\Q}{\mathbb{Q}}

\newcommand{\U}{\mathcal{U}}
\newcommand{\Z}{\mathbb{Z}}

\newcommand{\R}{\mathbb{R}}

\newcommand{\C}{\mathbb{C}}

\newcommand{\Hom}{{\rm Hom}\,}

\def\M{{\rm M}}
\def\G{{\rm G}}
\def\SL{{\rm SL}}

\def\GSp{{\rm GSp}}
\def\PGSp{{\rm PGSp}}
\def\Sp{{\rm Sp}}
\def\GU{{\rm GU}}
\def\U{{\rm U}}
\def\GO{{\rm GO}}
\def\GL{{\rm GL}}
\def\PGL{{\rm PGL}}
\def\GSO{{\rm GSO}}

\def\SO{{\rm SO}}
\def\O{{\rm O}}

\def\tr{{\rm tr\,}}

\begin{document}

\title{Bessel Models for \GSp(4)}
\author{Dipendra Prasad and Ramin Takloo-Bighash}

\vspace{2mm}

\dedicatory{\hfill To Steve Gelbart}
\address{D. P.: School of Mathematics, Tata Institute of Fundamental
Research, Colaba, Mumbai-400005, INDIA}
\email{dprasad@math.tifr.res.in}
\address{R. T.-B.: Department of Mathematics, Statistics, and Computer Science, University of Illinois at Chicago,
851 S. Morgan St, Chicago, IL 60607} \email{rtakloo@math.uic.edu}
\thanks{
The first author thanks the Institute for Advanced Study, as well as the
University of California at San Diego, where
this work was done, and gratefully acknowledges receiving support through
 grants to the Institute by the Friends of
the Institute, and the von Neumann Fund. The second author is
partially supported by a Young Investigator's Award from the NSA and
a grant from the NSF. We are grateful to W.T. Gan for many
suggestions; in fact, his help has been invaluable. Both of the
authors have received warm encouragement from Steve Gelbart at
various stages and would like to dedicate this work to him.}
\maketitle
\begin{abstract}
Methods of theta correspondence are used to analyse local and global 
Bessel models  for $\GSp_4$ proving a conjecture of Gross and Prasad 
which describes these models in terms of  local epsilon factors in the 
local case, and  the non-vanishing 
of central critical $L$-value in the global case. 
\end{abstract}

\tableofcontents

\section{Introduction}
In this paper we will use the methods of theta correspondence to
prove certain  local and global conjectures of Gross-Prasad for the
pair $(\SO(2), \SO(5)) $ in some cases which amounts to the
understanding of Bessel models for representations of $\SO(5)$ by
reducing this question to one about branching laws  involving
simpler pairs for which the analogous question is known. These 
conjectures  relate the 
existence of Bessel models (which are certain Fourier coefficients) 
to certain local epsilon factors in the local case, and to the non-vanishing 
of certain central critical $L$-value in the global case. 
Instead of
$\SO(5)$ we will consider the related group $\GSp_4$ in this paper.
This gives first nontrivial evidence to the conjectures of
Gross-Prasad in which the subgroup considered is neither reductive,
nor unipotent.

Among the earliest manifestations of the methods that we follow in
this paper is the work of Waldspurger on Shimura correspondence in
the late 70's relating period integral of automorphic forms on
$\PGL_2$ over tori to Fourier coefficients of automorphic forms on
the metaplectic $\SL_2$-- both being related to twisted $L$-values
at $1/2$.

Let $W$ be a four dimensional symplectic vector space over a local
field $k$. Let $W_1$ be a maximal isotropic subspace of $W$. Let
$G=\GSp(W)$ denote the symplectic similitude group of $W$, and $P$
the parabolic subgroup of $G$ consisting of elements of $G$ which
take $W_1$ into itself. The group $P$ is the so-called Siegel
parabolic which has the Levi decomposition $P=MN$, where $M \cong
\GL_2 \times {\Bbb G}_m $, and $N$ is an  abelian group which can be
identified to $2 \times 2$ symmetric matrices ${\rm Sym}_2(k)$ over
$k$. An element $(g,\lambda) \in  \GL_2 \times {\Bbb G}_m $ acts on
$n \in N$ by $\lambda gn{}^tg$.  It follows that the stabiliser in
$M$ of a nondegenerate symmetric matrix  in $N$ can be identified to
 the normalizer of a Cartan subgroup of $\GL_2$. Fix $\psi_{00}: k \rightarrow
{\Bbb C}^*$ to be a nontrivial additive character of $k$. Let
$\psi_0:N \cong {\rm Sym}_2(k) \rightarrow {\Bbb C}^*$  be
$\psi_0(n) = \psi_{00}({\rm trace}( n))$. Any character of $N$ is of
the form $\psi(n) = \psi_0(sn)$ for some $s \in {\rm Sym}_2(k)$, and
the corresponding  subgroup $T = T_s$ of $\GL_2(k)$ is
\begin{equation*}
T = \{ g \in \GL_2(k)\, \vert \, \,^t g s g = \det g. s \}.
\end{equation*}
We consider $T$ as a subgroup of $\GSp_4(k)$ via
\begin{equation*}
g \mapsto \begin{pmatrix} g \\ & \det g. \,^t g ^{-1}
\end{pmatrix}.
\end{equation*}

Let $\pi$ be an irreducible admissible representation of $\GSp(W)$.
Let $\pi_\psi$ denote the largest quotient of $\pi$ on which $N$
operates by $\psi$. Clearly $\pi_\psi$ is a representation space for
the subgroup $M^\psi$ of  $M$ which stabilizes $\psi$. We will
consider $\psi = \psi_0(gn)$ corresponding to a $g \in {\rm
Sym}_2(k)$ with $\det g \not = 0$. For such $\psi$, $M^\psi$ is
isomorphic to the normaliser $N(T)$ of a Cartan subgroup $T$ of
$\GL_2$. The question that we study in this paper is the structure
of $\pi_\psi$ as a module for $T$, called the Bessel model of $\pi$,
both locally as well as globally for certain representations $\pi$
of $\GSp(W)$ which will be the set of `most' irreducible
representations of $\GSp(W)$; in the global context, we will be able
to answer this question only for a small class of  automorphic
representations of $\GSp(W)$.

We will work simultaneously with the rank 1 form of $\GSp_4(k)$,  to
be denoted by $\GSp_D(4)$, and  defined using a quaternion division
algebra $D$ over $k$  as
$$\left \{ g \in \GL_2(D) | g
\left ( \begin{array}{cc} 0 & 1 \\ 1 & 0 \end{array} \right ){}^t\bar{g}
= \lambda \cdot  \left ( \begin{array}{cc} 0 & 1
\\ 1 & 0 \end{array} \right ) , \lambda \in k^* \right \}$$
where for $g = \left ( \begin{array}{cc} a & b \\ c & d \end{array} \right )
\in \GL_2(D), {}^t\bar{g} = \left ( \begin{array}{cc} \bar{a} & \bar{c} \\
\bar{b} & \bar{d} \end{array} \right )$, and where $a \rightarrow
\bar{a}$ denotes the standard involution on $D$. The group
$\GSp_D(4)$ contains the Siegel parabolic whose unipotent radical is
the group of matrices
$$\left ( \begin{array}{cc} 1 & n \\ 0 & 1 \end{array} \right )$$
where $n \in D$ with $n + \bar{n} = 0$, and the Levi subgroup is
isomorphic to $D^* \times k^*$ embedded in $\GSp_D(4)$ as
$$\left ( \begin{array}{cc} d & 0 \\ 0 & t d \end{array} \right )$$
for $d \in D^*,$ and $t \in k^*$.

We begin by recalling  the following multiplicity 1 theorem of
Novodvorsky extended in two ways. First we consider $\GSp_4(k)$
instead of his $\PGSp_4(k)$, and then we also consider rank 1 form
of $\GSp_4(k)$. Both these are standard extensions of the arguments
in Novodvorsky.

\begin{thm}\label{thm1.1}
Let $\pi$ be an irreducible admissible representation of either
$\GSp_4(k)$, or $\GSp_D(4)$ with Siegel parabolic $P=MN$. 
Let $K$ be a quadratic algebra over
$k$, and $\chi$ a character of $K^*$. Let $\psi: N(k) \rightarrow
{\Bbb C}^*$ be a nondegenerate character of $N(k)$ centralized by
$K^*$, so that one can construct a one dimensional representation of
$R= K^* N(k)$ which is $\chi $ on $K^*$, and $\psi$ on $N(k)$, which
will also be denoted by $\chi$ as $\psi$ will be kept fixed in this
paper. Then
$${\rm dim}{\rm Hom}_R(\pi, \chi) \leq 1.$$

\end{thm}

\begin{rem} If ${\rm Hom}_R(\pi, \chi) \neq 0,$ then the representation
$\pi$ is said to have Bessel model for the character $\chi$ of
$K^*$.
\end{rem}

Before we  state the Gross-Prasad conjecture in this context, we remind
ourselves that the Langlands parameter of a representation $\pi$
 of $\GSp_4(k)$
is a  representation
$$\sigma_{\pi}: W'_k\rightarrow \GSp_4({\Bbb C})$$ where $W'_k$
is the Weil-Deligne group of $k$ which we take to be $W'_k =W_k
\times \SL_2({\Bbb C})$. These  have been constructed in a recent
paper of Gan and Takeda who have also defined a notion of
$L$-packets for representations of $\GSp_4(k)$ (of size 1 or 2)
which is what we will use in this paper.

\vspace{2mm}

\begin{conj}\label{conj1}

Let $K$ be a quadratic algebra over $k$ such that $K^* \subset
\GL_2(k)$ is contained in the centralizer of a nondegenerate
character $\psi: N(k) \rightarrow {\Bbb C}^*$.  Let $\chi: K^*
\rightarrow {\Bbb C}^*$ be a character, and let
 $\chi$ also denote the
corresponding character of $R=K^*N(k)$ which is $\chi$ on $K^*$ and
$\psi$ on $N(k)$. Let $\{\pi\}$ be an irreducible, admissible
generic $L$-packet of representations of $\GSp_4(k)$ with Langlands
parameter  $\tau$. Assume that the central character of $\{ \pi \}$
is the same as $\chi|_{k^*}$. Let  $\GSp_D(4)$ be the rank 1 form of
$\GSp_4(k)$, and $\{\pi'\}$  an irreducible, admissible generic
$L$-packet of representations of $\GSp_D(4)$ with Langlands
parameter $\tau$. (So $\{\pi'\}$ might be an empty set.) Then there
is at most one  representation $\pi \in \{\pi\}$  with ${\rm
Hom}_R(\pi_{\psi}, \chi) \not = 0$, and there is one if and only if
$\epsilon( \tau \otimes {\rm { Ind}_K^k}(\chi^{-1})) = 1. $
Similarly, there is at most one  representation $\pi' \in \{\pi'\}$
with ${\rm Hom}_R(\pi'_{\psi}, \chi) \not = 0$, and there is one  if
and only if $\epsilon( \tau \otimes {\rm { Ind}_K^k}(\chi^{-1})) =
-1.$ Furthermore, if $\{\pi\}$ or $\{\pi'\}$ consisted of more than
one element, then the parameter $\tau$ with values in $\GSp_4({\Bbb
C})$ becomes a sum of representations $\tau = \tau_1 \oplus \tau_2$,
and one can make precise which element of the $L$-packet, $\{\pi\}$
or $\{\pi'\}$ carries the $\chi$-invariant form on the subgroup $R$.

\end{conj}

The proof of this conjecture will be via the methods of theta
correspondence with $\GO(4)$ for which the conjectures of
Gross-Prasad (involving only reductive groups) are known. What makes 
this possible is the fact 
that in odd residue characteristic, every irreducible supercuspidal 
representation
of $\GSp_4(k)$ arises as a theta lift of a representation of
$\GO_4(k)$ for some four dimensional quadratic space. Having proved the 
conjecture for
$\GSp_4(k)$, one can actually 
bootstrap to deduce something for $\GL(4)$,
which we state now. (This theorem is proved in this paper for all
representations of $\GL_4(k)$ whose Langlands parameter is not of
the form $\sigma \otimes {\rm St}_2$ where $\sigma$ is a parameter
of a supercuspidal representation of $\GL_2(k)$ with a nontrivial
self-twist, and ${\rm St}_2$ is the parameter of the Steinberg
representation of $\PGL_2(k)$.)

\begin{thm}\label{thm1.3}
Let $D$ be a quaternion division algebra over a local field $k$, $K$
a quadratic separable algebra over $k$. Let $\pi$ be an irreducible,
admissible, generic representation of $\GL_4(k)$ with central
character $\omega_\pi$. Let $\chi$ be a character of $K^*$, also
considered as a character of $\GL_2(K)$ via the determinant map
$\det: \GL_2(K)  \rightarrow K^*$, such that $\chi^2|_{k^*} =
\omega_{\pi}$. Then the character $\chi$ of $\GL_2(K)$ appears as a
quotient in $\pi$ restricted to $\GL_2(K)$  if and only if

\begin{enumerate}

\item The Langlands parameter of $\pi$ takes values in $\GSp_4({\Bbb C})$ with similitude factor $\chi|_{k^*}$.

\item The epsilon factor
$\epsilon( \pi \otimes {\rm { Ind}_K^k}(\chi^{-1})) = 1. $
\end{enumerate}

Similarly, assuming that $\pi$ can be transferred to a
representation $\pi'$ of $\GL_2(D)$, and that $K$ is a quadratic
field extension of $k$ so that $\GL_2(K)$ embeds into $\GL_2(D)$,
the character $\chi$ of $\GL_2(K)$ appears in $\pi'$ restricted to
$\GL_2(K)$ as a quotient if and only if

\begin{enumerate}

\item The Langlands parameter of $\pi$ takes values in $\GSp_4({\Bbb C})$ with
similitude factor $\chi|_{k^*}$.

\item The epsilon factor
$\epsilon( \pi \otimes {\rm { Ind}_K^k}(\chi^{-1})) = -1. $
\end{enumerate}

\end{thm}

In this paper, we will be considering representations of $\GSp(W)$
which arise by theta lifting from representations of $\GSO(V)$,
where $V$ is a quadratic space of dimension 4. Recall that $\GO(V)$
contains a subgroup that we denote by $\GSO(V)$, of index 2, which
is the connected component of the identity of $\GO(V)$. It is
well-known that $\GSO(V)$ has the structure of one of the following
groups:

\begin{enumerate}
\item $\GSO(V^s) \cong (\GL_2(k) \times \GL_2(k) )/\Delta k^*.$

\item  $\GSO(V^a) \cong (D^* \times D^*)/\Delta k^*. $

\item $\GSO(V^d) \cong [\GL_2(E) \times k^*]/\Delta E^*,$
\end{enumerate}
where $\Delta k^* = k^*$ sits as $(t,t^{-1})$, and
$\Delta E^* = E^*$ sits inside $\GL_2(E) \times k^*$ via its natural
embedding in $\GL_2(E)$, and in $k^*$ by the norm mapping; further,
we have used $V^s$ to denote the unique four dimensional split
quadratic space, $V^a$ to denote the unique anisotropic quadratic
space of dimension 4, and $V^d$ is one of the two quadratic spaces
of rank 1 with discriminant algebra $E$, a quadratic field extension
of $k$.

Let $\tau$ be an irreducible representation of $(\GL_2(k) \times \GL_2(k)
)/ \Delta k^*$ of the form $\tau \cong \tau_1 \boxtimes \tau_2$
for irreducible admissible representations $\tau_1$ and $\tau_2$ of
$\GL_2(k)$ such that  their central characters are the same. 
An irreducible representation of $(D^* \times D^*)/\Delta
k^*$ will also be written as $\tau \cong \tau_1 \boxtimes
\tau_2$. For $\tau$ an irreducible representation of $\GSO(V^s)$,
let $\pi_1 = \Theta(\tau)$ be the theta lift of $\tau$. If $\tau$ is
a discrete series representation of $\GSO(V^s)$, we let $\tau'$
denote the representation of $\GSO(V^a)$ obtained by the
Jacquet-Langlands correspondence. Let $\pi_2 = \Theta(\tau')$ be the
theta lift of $\tau'$.

We  define $\GO_D(4)$ to be
$$\left \{ g \in \GL_2(D) | g
\left ( \begin{array}{cc} 0 & 1 \\ -1 & 0 \end{array} \right )
{}^t\bar{g}
= \lambda \cdot  \left ( \begin{array}{cc} 0 & 1
\\ -1 & 0 \end{array} \right ) \right \}.$$
It is clear that $\GO_D(4)$ contains the matrices
$$\left ( \begin{array}{cc} d & 0 \\ 0 & d
\end{array} \right )$$ for $d \in D^*$,
as well as $\GL_2(k)$. From this it is easy to see that in fact
$\GO_D(4) \cong [D^* \times \GL_2(k)]/ \Delta k^*$. One can
define $\GO_D(4)$ more generally by replacing the matrix $\left (
\begin{array}{cc} 0 & 1 \\ -1 & 0 \end{array} \right )$ by a general
skew-Hermitian matrix in $\GL_2(D)$.  By generality, it is known
that over a non-Archimedean  local field $k$, a skew-Hermitian form
in $\GL_2(D)$ is classified by its discriminant which is an element
of $k^*/k^{*2}$ which since we are classifying inner forms of
$\GSO(V^s)$, we take to be 1, hence the skew-Hermitian form $\left (
\begin{array}{cc} 0 & 1 \\ -1 & 0 \end{array} \right ) $ is the only
relevant one at this point. For construction of inner forms of
$\GSO(V^d)$, other discriminants will become relevant.

It is well-known that $(\GSp_4, \GSO(V^s))$, $(\GSp_4, \GSO(V^a))$,
as well as  $(\GSp_D(4), \GO_D(4))$ can be considered as dual
reductive pairs, and one can use theta liftings to obtain
representations of $\GSp_D(4)$ from, in particular,
 those of $\GSO(V^s)$, and
$\GSO(V^a)$. Define $\pi_3 = \Theta(\tau)$, and $\pi_4 =
\Theta(\tau')$; where given an irreducible representation $\tau_1
\boxtimes \tau_2$ of $\GL_2(k) \times \GL_2(k)$, there are two
representations of $\GO_D(4) \cong D^* \times \GL_2(k)$ 
by using the
Jacquet-Langlands correspondence either on $\tau_1$ or on $\tau_2$;
if one of $\tau_1$ or $\tau_2$ is not a discrete series
representation, then one or both of $\tau$ and $\tau'$ could be
taken to be 0. If $\tau_1 = \tau_2$, then the representation
$\tau_1 \boxtimes JL(\tau_1)$ of $\GO_D(4) \cong \GL_2(k)  \times D^*$ 
does not lift to $\GSp_D(4)$.

We now state  one of our local theorems which is for those
representations of $\GSp_4(k)$ which arises from theta lifting from
$\GSO(V^s)$, or $\GSO(V^a)$ and for which we simultaneously have to
consider representations of $\GSp_D(4)$ arising from $\GO_D(4)$. We
will not state here what we prove for representations of $\GSp_4(k)$
arising from the dual pair $\GSO(V^d)$ for which we refer to Section
\ref{sect:applications}. There are other representations which
cannot be handled by the method of theta correspondence which in the
odd residue characteristic are certain subquotients of principal
series representations for which we use a combination of the method
of theta correspondence and the Mackey theory for the full induced
representations, taken up later in the paper.

\begin{thm}\label{thm1.4}
Let $K$ be a quadratic algebra over $k$ such that $K^* \subset
\GL_2(k)$ is contained in the centralizer of a nondegenerate
character $\psi: N(k) \rightarrow {\Bbb C}^*$.  Let $\chi: K^*
\rightarrow {\Bbb C}^*$ be a character, and let
 $\chi$ also denote the
corresponding character of $R=K^*N(k)$. Let $\tau_1$ and $\tau_2$ be
irreducible, admissible representations of $\GL_2(k)$ with the same
central characters which is $\chi|_{k^*}$, and $\pi_1$ the theta
lift to $\GSp_4(k)$ of the representation $\tau_1 \boxtimes \tau_2$ of
$\GSO(V^s)$. Let $\pi_2$ (resp. $\pi_3,\pi_4$) be representations of
$\GSp_4(k)$  (resp.  $\GSp_D(4)$) as defined earlier via theta lifts
from $\GSO(V^a)$ (resp. $\GO_D(4)$). (Observe that if $\tau_1$ or
$\tau_2$ is not a discrete series representation of $\GL_2(k)$, then
some of the representations $\pi_2, \pi_3,\pi_4$ may not be
defined; if $\tau_1 = \tau_2$, then $\pi_3=\pi_4=0$.) 
For the character $\chi$ of $K^*$, let ${\rm
Ind}_K^k(\chi)$ denote the two dimensional representation of the Weil
group of $k$ obtained by inducing the character $\chi$ of $K^*$.
With this  notation, we have
\begin{enumerate}
\item ${\rm Hom}_R(\pi_{1,\psi}, \chi) \not = 0$ if and only if
$$\epsilon( \tau_1 \otimes {\rm { Ind}_K^k}(\chi^{-1})) = \epsilon(\tau_2\otimes
{\rm { Ind}_K^k}(\chi^{-1})) = \omega_{K/k}(-1)\chi(-1).$$

\item ${\rm Hom}_R(\pi_{2,\psi}, \chi) \not = 0$ if and only if
$$\epsilon( \tau_1 \otimes {\rm { Ind}_K^k}(\chi^{-1})) = \epsilon(\tau_2\otimes
{\rm { Ind}_K^k}(\chi^{-1})) = -\omega_{K/k}(-1)\chi(-1).$$

\item ${\rm Hom}_R(\pi_{3,\psi}, \chi) \not = 0$ if and only if
$$\epsilon( \tau_1 \otimes {\rm { Ind}_K^k}(\chi^{-1})) = -\epsilon(\tau_2\otimes
{\rm { Ind}_K^k}(\chi^{-1})) = \omega_{K/k}(-1)\chi(-1).$$

\item ${\rm Hom}_R(\pi_{4,\psi}, \chi) \not = 0$ if and only if
$$\epsilon( \tau_1 \otimes {\rm { Ind}_K^k}(\chi^{-1})) = -\epsilon(\tau_2\otimes
{\rm { Ind}_K^k}(\chi^{-1})) = -\omega_{K/k}(-1)\chi(-1).$$

\end{enumerate}

\end{thm}

Here is the corresponding  global theorem.

\begin{thm}\label{thm1.5}

Let $D$ be a quaternion  algebra over a number field $F$, with the
adele ring ${\Bbb A}_F$. Let $\Pi_1$ and $\Pi_2$ be two automorphic
representations of $D^*({\Bbb A}_F)$ with the same central
characters, so that $\Pi_1 \boxtimes \Pi_2$ can be considered to be an
automorphic representation on the corresponding orthogonal group
$\GSO_4({\Bbb A}_F)$ defined by the reduced norm on $D$. Let $\Pi$
be the theta lift to $\GSp_4({\Bbb A}_F)$ of $\Pi_1 \boxtimes\Pi_2$ on
 $\GSO_4({\Bbb A}_F)$. Let $E$ be a separable quadratic algebra over $F$,
and $\chi$ a Gr\"ossencharacter on ${\Bbb A}_E^*$ whose restriction 
to ${\Bbb A}_F^*$ is the central character of $\Pi$. Let $\psi: 
N({\Bbb A}_F)/N(F) \rightarrow {\Bbb C}^*$ be a character which is normalized
by ${\Bbb A}_E^*$, and hence $(\chi,\psi)$ gives rise to a character
 on $R({\Bbb A}_F)= {\Bbb A}_E^* N({\Bbb A_F})$ which we will
abuse notation to denote simply by $\chi$ as $\psi$ is considered fixed. 
Then the period integral on $\Pi$ taking $f \in \Pi$ to
$$\int_{R(F){\Bbb A}_F^*\backslash R({\Bbb A}_F)} f(g) \chi^{-1}(g) dg ,$$
is not identically zero if and only if the period integrals,
$$\int_{E^*{\Bbb A}_F^*\backslash {\Bbb A}_E^*} f_1(g) \chi^{-1}(g) dg ,$$
and
$$\int_{E^*{\Bbb A}_F^*\backslash {\Bbb A}_E^*} f_2(g) \chi^{-1}(g) dg $$
on $\Pi_1$ and $\Pi_2$ respectively are not identically zero; in particular,
by Waldspurger, if the period integral on
${R(F){\Bbb A}_F^*\backslash R({\Bbb A}_F)}$ of functions in $\Pi$ is not
identically zero,
then
$$ L(\frac{1}{2}, BC_E(\Pi_1) \otimes \chi^{-1}) \not = 0,$$
and
$$ L(\frac{1}{2}, BC_E(\Pi_2) \otimes \chi^{-1}) \not = 0,$$
where $BC_E$ denotes the base change to $\GL_2(E)$ of an automorphic
form on $\GL_2(F)$.

Further, for automorphic representations $\Pi_1$ and $\Pi_2$ on
$\GL_2({\Bbb A}_F)$, if $ L(\frac{1}{2}, BC_E(\Pi_1) \otimes \chi^{-1})
\not = 0,$ and $ L(\frac{1}{2}, BC_E(\Pi_2) \otimes \chi^{-1}) \not
= 0,$  Waldspurger's theorem gives quaternion algebras $D_1$ and
$D_2$ over $F$, and an automorphic representation of $(D_1^* \times
D_2^*)/\Delta{\Bbb G}_m$ such that the corresponding period
integrals on $E^*{\Bbb A}_F^*\backslash {\Bbb A}_E^*$ are nonzero.
Given $D_1$ and $D_2$, quaternion algebras over the number field
$F$, let $D_1 \otimes D_2 \cong M_2(D)$. Taking the tensor product
of canonical involutions on $D_1$ and $D_2$, we get an involution on
$M_2(D)$, and hence there is a hermitian form on a 2 dimensional
vector space over $D$ such that the corresponding $\GSO_D(4)= [D_1^*
\times D_2^*] / \Delta {\Bbb G}_m$. Define $\GSp_4(D)$ using this
$D$, and construct a representation of $\GSp_4(D_{\Bbb A})$  via
theta lifting. Then for this automorphic representation,
  say $\tilde{\Pi}$  on
$\GSp_D(4)$, the corresponding period integral of functions $f$ in
$\tilde{\Pi}$,
$$\int_{R(F){\Bbb A}_F^*\backslash R({\Bbb A}_F)} f(g) \chi^{-1}(g) dg ,$$
is not identically zero (in particular $\tilde{\Pi}$ is not zero).

\end{thm}

\begin{rem} There is a considerable amount of {\it Geometric Algebra}
especially using {\it exceptional isomorphisms} of low rank groups
in this paper (such as $\SO(5)$ being isomorphic to $\PGSp(4)$ or
$\SO(6)$ being related to $\SL(4)$, the structure of their inner
forms which usually has different constructions for the two groups
involved, and the relation of their subgroups under this
isomorphism). This fits rather nicely  to yield exactly what is
needed for the similitude groups being considered (such as
$\GSO(2)$, which is much preferred over $\SO(2)$).
\end{rem}

\begin{rem} It may be noted that  besides its
intrinsic interest, as Bessel models are usually nonzero for some
choice of $\chi : K^* \rightarrow {\Bbb C}^*$, they can be used in
developing the theory of $L$-functions for $\GSp_4(k)$ as in the
early work of Novodvorsky and Piatetski-Shapiro, extending
considerably the scope of the theory of $L$-functions based on
genericity hypothesis. \end{rem}

We end this introduction by briefly recalling  how the proof of the
main theorems are obtained. The proof of the p-adic theorem is achieved
in two parts. First we  deal  with all subquotients of principal
series by standard Mackey theory (restriction of an induced
representation to a subgroup) which calculates the Bessel models of
representations involved, and then a separate calculation on root
numbers combined with the theorem of Saito and Tunnell about
characters of $GL_2(k)$ yields conjecture \ref{conj1} in these
cases. Some extra complication arises when the representation is not
fully induced, but this is a small number of cases, in fact two
kinds of representations that really matter; one of these can be
handled by theta methods, the other remains as a problem as pointed
out in the body in the paper. One is then left with supercuspidals
which in odd residue characteristic are obtained by theta lifting
from $\GO(4)$. These are dealt with by the methods related to the Weil
representation. We will also prove a theorem for the archimedean
case dealing with discrete series Vogan packets. The sufficiency of
the conditions of the archimedean theorem follows from the work of
Takloo-Bighash for rank two case. In this case one uses the results
of Kolk and Varadarajan to establish the necessity of the
conditions. In the rank one case, the results follow directly 
from Wallach's
results. We will also include a theorem on the existence of Bessel
functionals in the global situation, and will explain how one can
draw local consequences from the global theorem.

\section{Bessel models for principal series representations}

The aim of this section is to calculate the twisted Jacquet functor
$\pi_\psi$ for a principal series representation $\pi$  of $\GSp_4(k)$ 
with respect to
a non-degenerate character $\psi: N \rightarrow {\Bbb C}^*$ given by
a symmetric matrix $A \in M_2(k)$ as $\psi(X)= \psi_0(tr(AX))$ for
$X \in N = {\rm Sym}_2(k)$.

We note that $\pi_\psi$ is a module for the subgroup $$M_\psi \cong
\{g \in \GL_2(k)| gA^tg= \lambda(g) A\}$$ which is the normaliser in
$\GL_2(k)$ of a certain torus $K_\psi^*$.

At certain points in the paper, we will find it more convenient to work 
with co-ordinates, for which we fix some notation here.

Let $W$ be a four dimensional symplectic vector space over a field $k$ 
with a fixed basis $\{e_1,e_2,e_3,e_4\}$, and a symplectic form
$<,>$ on $W$ such that $<e_1,e_3>=-<e_3,e_1> = 1$, 
$<e_2,e_4>=-<e_4,e_2> =1$, and all other products among these basis
vectors to be zero.

\subsection{Siegel Parabolic}

We begin with the case when $\pi$ is induced from the Siegel
parabolic $P$ from an irreducible representation $\rho$ of the
Levi subgroup $M$ of $P=MN$.

As usual, let $P$ be the Siegel parabolic stabilizing the isotropic
subspace $W= \{e_1,e_2\}$ with $M$ the stabiliser of the isotropic
subspaces $W=\{e_1,e_2\}$ and $W'=\{e_3,e_4\}$. The calculation of
the twisted Jacquet functor will depend on the understanding of the
double coset $P\backslash G/P$ with $G= \GSp_4(k)$, which is the
same as $G\backslash [G/P \times G/P]$, or the orbit of $\GSp_4(k)$
on pairs of maximal isotropic subspaces. It is easy to see that
there are three orbits of pairs of isotropic subspaces $(W_1,W_2)$:

\begin{enumerate}

\item $W_1 = W_2$; in this case we take $W_1=W_2= W$.

\item $W_1\cap W_2 = \{0\}$; in this case we take $W_1=W$, and
$W_2 = W'$.

\item $W_1 \cap W_2$ is 1-dimensional; in this case we take $W_1=W$,
and $W_2 = \{e_1,e_4\}$.

\end{enumerate}

As $W_1$ is chosen to be $W$ in all the three cases, the stabiliser
in $\GSp_4(k)$ of the pair of isotropic subspaces $(W_1,W_2)$ is a
subgroup of $P$ which is the following subgroup $H_i$ of $P$ in the
three cases:

\begin{enumerate}
\item $H_1=P$.

\item $H_2=M$.

\item $H_3$ containing the unipotent group
$ \left ( \begin{array}{cc} x & y
\\ y & 0 \end{array} \right ) \in {\rm Sym}^2(k)  \subset N$.

\end{enumerate}

From the Mackey theory, it follows that the representation $\pi =
{\rm ind}_P^G\rho$
restricted to $P$ is obtained by gluing the following three
representations:

\begin{enumerate}

\item $\rho$.

\item ${\rm ind}_M^P\rho$.

\item ${\rm ind}_{H_3}^P\rho|_{H_3}$.

\end{enumerate}

(The discriminant function $\delta_P$ used for normalized induction
can be seen to be trivial on $M_\psi$, hence in light of the eventual
answer, one can ignore $\delta_P$ in what follows.)

We observe that since the representation $\rho$ of $M$ is extended
to $P$ trivially across $N$, for a nondegenerate character $\psi$
of $N$, $\rho_\psi=0$ in case $(i)$.

For case $(iii)$, as
  $$ \left ( \begin{array}{cc} a & b
\\ b & c \end{array} \right )
\left ( \begin{array}{cc} x & y
\\ y & 0 \end{array} \right )
=\left ( \begin{array}{cc} ax+by & ay
\\bx+cy & by \end{array} \right ),$$
  $$\psi\left[ \left ( \begin{array}{cc} a & b
\\ b & c \end{array} \right )
\left ( \begin{array}{cc} x & y
\\ y & 0 \end{array} \right ) \right] = \psi_0(ax+2by),$$
it follows that if the character $\psi$ of $N$ were to be trivial
on the subgroup $N \cap H_3$, $a=b=0$, and hence $\psi$
will not be a nondegenerate character. Noting that $N$ is a
normal subgroup of $P$, it follows that any character of $N$
appearing in case $(iii)$ is degenerate, and hence case $(iii)$
does not contribute to the twisted Jacquet functor.

In case $(ii)$, since $P \cong MN$ with $N$ normal, we get $
{\rm ind}_M^P\rho \cong \rho \otimes {\rm ind}_M^P {\bf 1}$ 
as $P$-modules. Hence,
  $$\left[ {\rm ind}_M^P\rho \right]_\psi  \cong \rho|_{M_\psi}$$
for any character $\psi$ of $N$. Thus we find that the twisted Jacquet
functor in the three cases is as follows:

\begin{enumerate}
\item 0.

\item $\rho|_{M_\psi}$.

\item 0.

\end{enumerate}

Therefore we have the following proposition.

\begin{prop}\label{prop2.1}
For a principal series representation $\pi$ of $G=\GSp_4(k)$ induced
from a representation $\rho$ of $P=MN$ of a Siegel parabolic,
$\pi_\psi \cong \rho$ restricted to $M_{\psi}$.
\end{prop}

Analogously, for $\GSp_D(4)$, we have the following.

\begin{prop}\label{prop2.2}
For a principal series representation $\pi$ of $G=\GSp_D(4)$ induced
from a representation $\rho$ of $P=MN$ of the unique parabolic 
of $\GSp_D(4)$(up
to conjugacy) with $M \cong D^* \times k^*$, $\pi_\psi \cong \rho$
restricted to $M_{\psi}$.
\end{prop}
\subsection{Klingen Parabolic}
We next direct our attention to the calculation of the twisted
Jacquet functor for representations induced from Klingen parabolic
$Q=M'N'$ which we take to be the stabiliser of the isotropic line
$\{e_1\}$. Once again, the restriction to $P$ of a representation
$\pi$ of $\GSp_4(k)$ induced from a representation $\rho$ of $M'$
extended trivially across $N'$ is obtained by gluing certain
representations indexed by double cosets $P\backslash \GSp_4(k)/Q$
which is the same as the $\GSp_4(k)$-orbits of a pair consisting of
a pair $(L,W)$ of one dimensional subspace $L$ of $V$, and a
two dimensional isotropic subspace $W$. There are two orbits:

\begin{enumerate}

\item $L \subset W$ in which case we take $L=\{e_1\}$, and $W=\{e_1,e_2\}$.

\item $L \not \subset W $ in which case we take $L=\{e_3\}$,
$W= \{e_1,e_2\}$.

\end{enumerate}

In case $(i)$, the part of the unipotent radical $N$ of $P$ which is
contained in the unipotent radical $N'$ of $Q$ is the set of matrices,
$$ \left ( \begin{array}{cc} 0 & y
\\ y & z \end{array} \right ) \in {\rm Sym}^2(k)  \subset N,$$

A calculation as done in case $(iii)$ of the principal series induced
from Siegel parabolic, it is easy to see that there are no
nondegenerate characters of $N$ trivial on
$$ \left ( \begin{array}{cc} 0 & y
\\ y & z \end{array} \right )
\in {\rm Sym}^2(k)  \subset N,$$
and therefore once again as $N$ is normal in $P$, it follows that
this double coset contributes nothing to the twisted Jacquet functor.

In the case $(ii)$, the stabiliser of the pair $(L,W)$ with
$L=\{e_3\}$, and $W=\{e_1,e_2\}$ is the subgroup

$$H= \left( \begin{array}{cccc} x_{11} & 0 & 0 & 0 \\
                            x_{21} &       x_{22} & 0      & x_{24} \\
                            0 &        0     & x_{33} & x_{34} \\
                            0 &         0 &     0     & x_{44}
\end{array} \right). $$

There are embeddings of $H$ in $Q= k^* \times \GL_2(k) \times N'$
with image $k^* \times B_2 \times \langle u \rangle$ where $B_2$
is the group of upper triangular matrices in $\GL_2(k)$, and $<u>$ is a 1
parameter subgroup in $N'$. In the embedding of $H$ in $P= k^*
\times \GL_2(k) \times N $, the one parameter subgroup $\langle u
\rangle$ goes to the upper triangular unipotent subgroup in 
$\GL_2(k)$, the unipotent radical of $B_2$ goes inside $N$ to a
1-dimensional subgroup that we denote by $N_0$, and the diagonal
torus to the diagonal torus in $\GL_2(k)$.

By Mackey theory, it follows that the restriction of $\pi$ to $P$ contains
${\rm ind}_H^P \rho'$ where $H$ can be taken to be $k^* \times B_2
\times N_0$ and the representation $\rho'$ is the restriction of
$\rho$ to the diagonal torus in $\GL_2(k)$ extended trivially across
the unipotent subgroup of $B_2$ to all of $B_2$. We assume now that
the representation $\rho$ of $\GL_2(k)$ is infinite dimensional, so
that (by Kirillov theory) its restriction to $B_2$ contains the
representation of $B_2$ which is obtained by inducing from a
character of the subgroup, $ZU$ of $B_2$,  consisting of central and
unipotent elements of $B_2$.

As the action of $K^*$ on $\GL_2(k)/P_1$ where $P_1$ is the subgroup of
$B_2$ consisting of elements of the form
$$ \left ( \begin{array}{cc} 1 & *
\\ 0 & * \end{array} \right ), $$
contains an open dense orbit,
it follows that for $R=K^*\cdot N$, 
$R\backslash P/H$ contains an open dense double coset
which in the case of $K$ a field is the unique
double coset. Thus the representation ${\rm ind}_H^P \rho'$
restricted to $R$ contains
$${\rm ind}_{R \cap H}^R \tilde{\psi} = {\rm ind}_{k^*N_0}^{K^*N}\tilde{ \psi}, $$
where $\tilde{\psi}$ is the character of $k^* N_0$ which is equal to the 
central character of $\pi$ restricted to $k^*$, and is 
the restriction  of the character
$\psi$ of $N$ to $N_0$, which can be checked to be nontrivial on $N_0$.

Thus its maximal quotient on which $N$ operates by $\psi$ is
 ${\rm ind}_{k^*N}^{K^*N} \omega {\psi} $ where  $\omega$ is the
central character of the representation $\rho$. We thus obtain
the following conclusion.

\begin{prop}
For a principal series representation $\pi$ of $G=\GSp_4(k)$ induced
from an infinite dimensional irreducible
 representation $\rho$ of $Q=M'N'$ of a Klingen
parabolic, $\pi_\psi$ restricted to $K^*= M_{\psi}$ has each and
every character of $K^*$ with the same central character as that of
$\rho$ appearing with multiplicity one as a quotient.
\end{prop}

\subsection{Degenerate principal series coming from Klingen parabolic}

In this section we modify the argument of the previous section to
calculate the $\psi$-Bessel model, for a non-degenerate character
$\psi$ of $N$, of a degenerate principal series representation of
$\GSp_4(k)$ induced from a one dimensional representation $\rho$ of
the Klingen parabolic $Q = M'N'$. The analysis of the previous
section gives the $\psi$-Bessel model of $\pi = {\rm
ind}_Q^{\GSp_4(k)} \rho$ as the $\psi$-Bessel model 
of ${\rm ind}_H^P \rho|_{H}$ 
where $H = B_2 \times k^* \times N_0$, 
with
$B_2$ the lower triangular subgroup of $\GL_2(k)$, and $N_0$ the one
parameter subgroup
 $ \left ( \begin{array}{cc} 0 & 0
\\ 0 & * \end{array} \right ). $ If follows that if $\pi$ has $\psi$-Bessel
model for $s =  \left ( \begin{array}{cc} a & b
\\ b & c \end{array} \right ) \in \GL_2(k), $ then
$$ {\rm tr}\left [ \left ( \begin{array}{cc} a & b
\\ b & c  \end{array} \right )  \left ( \begin{array}{cc} 0  & 0
\\ 0 & * \end{array} \right ) \right ] \equiv 0, {\rm~~ or~~} c =0.  $$
This means that if $\pi$ has $\psi$-Bessel model corresponding to the
symmetric matrix
$s =  \left ( \begin{array}{cc} a & b
\\ b & c \end{array} \right ), $ $c$ must be zero. For such symmetric
matrices, the corresponding quadratic form is split, and hence we
deduce that $\pi$ has $\psi$-Bessel model only for $K$ defined by a split
quadratic algebra, in which case $K \cong k \oplus k$.

Fixing now the character $\psi \left ( \begin{array}{cc} x_{11} & x_{12}
\\ x_{12} &  x_{22} \end{array} \right ) = \psi_{00}(x_{12}), $
of $N$ which is trivial on $N_0$, and which is stabilized by the
diagonal split torus $T$ in $\GL_2(k)$, it is easy to calculate
the twisted Jacquet module for the character $\psi$ of $N$ of the
induced representation  ${\rm ind}_H^P \rho|_{H}$ 
with $H = B_2 \times k^* \times N_0$, to conclude the following
proposition.

\begin{prop}\label{prop2.4}
Let $\rho$ be a one dimensional representation of the Klingen
parabolic $Q$ of $\GSp_4(k)$, and $\pi = {\rm ind}_Q^{\GSp_4(k)}
\rho,$ the corresponding principal series representation. Then $\pi$
has Bessel models for a quadratic algebra $K$ if and only if $K = k
\oplus k$, and in which case it has exactly one dimensional space of
Bessel models for the character of $K^*$ obtained by restriction
of $\rho$ to $K^*$ which sits inside the Levi subgroup of $Q$.

\end{prop}

\section{Conjecture \ref{conj1} for irreducible principal series}

Having done the calculation of Bessel models for principal series
representations in the last section, we are now in position to
verify Conjecture \ref{conj1} for all generic $L$-packets of
$\GSp_4(k)$ which arise as subquotients of a principal series
representation induced from either of the two maximal parabolics. A
principal series representation of $\GSp_4(k)$ which is induced from
the Borel subgroup will be considered to be induced from the Siegel
parabolic. We will assume in this section that multiplicity one
theorem, i.e., Theorem \ref{thm1.1}, holds.

Our calculations when the principal series representation is irreducible
will be relatively straightforward. When the principal series representation
has more than one irreducible factor, the methods of the previous
section will give information only about a sum of multiplicities
of Bessel models for the various representations involved. In these cases
we will have to combine this information (about sum of multiplicities)
with an information that we deduce from the method of theta correspondence
developed later in the paper which yields the multiplicity of Bessel
model for the smaller (of the two representations involved), and thus
we succeed in the calculation of the multiplicities of the Bessel
models of the two representations.

Let us begin by stating the Langlands parameters of principal series
representations, and then do the relevant local root number
calculations.

\subsection{Siegel parabolic}

Let $P=MN$ be a Siegel parabolic with $M \cong  \GL_2(k) \times {\Bbb
G}_m$. Let $\pi_1 \boxtimes \mu$ be an irreducible representation of $M$,
giving rise to the principal series representation $\pi$ of
$\GSp_4(k)$ by parabolic induction. Then the Langlands parameter of
the representation $\pi$ of $\GSp_4(k)$ is a representation
$$\sigma_{\pi}: W'_k\rightarrow \GSp_4({\Bbb C})$$ where $W'_k$
is the Weil-Deligne group of $k$ which we take to be $W'_k =W_k
\times \SL_2({\Bbb C})$. Assuming that the Langlands parameter of
the representation $\pi_1$ of $\GL_2(k)$ is $\sigma_1$, and that of
$\mu$ is $\mu$ itself, we have
$$\sigma_\pi = \mu \sigma_1 \oplus (\mu \oplus \mu\det \sigma_1).$$
We note that the Langlands parameter of an irreducible
 principal series representation
of $\GSp_4(k)$ arising from parabolic induction of a representation
of the Siegel parabolic takes values in (Levi subgroup of) the
Klingen parabolic of $\GSp_4({\Bbb C})$.   (As a check on the
Langlands parameter, which the authors did not get right the first
time, one notes that the twisting by a character
$\chi:k^*\rightarrow {\Bbb C}^*$, thought of as a character of
$\GSp_4(k)$ through the similitude factor, takes the principal
series corresponding to $(\pi_1,\mu) $ to $(\pi_1,\mu \chi)$, and
this on Langlands parameter is supposed to be twisting by $\chi$.)

The central character $\omega_\pi$ of $\pi$ is the same as the
similitude character of $\sigma_\pi$ which is $\mu^2\det \sigma_1$.
Therefore the characters $\chi$ of $K^*$ appearing in Conjecture
\ref{conj1} have the property that $\chi|_{k^*}= \mu^2 \det
\sigma_1$, and these are the only characters that we will consider
in what follows.

By the standard properties of the local root numbers, it follows that
for $\sigma_\pi$ as above,
\begin{eqnarray*}
& & \epsilon(\sigma_\pi \otimes {\rm Ind}_K^k(\chi^{-1})) \\
&=&
\epsilon(\mu \sigma_1 \otimes {\rm Ind}_K^k(\chi^{-1})) \cdot
\epsilon(\mu \otimes {\rm Ind}_K^k(\chi^{-1})) \cdot
  \epsilon(\mu\det \sigma_1 \otimes {\rm Ind}_K^k(\chi^{-1})).
\end{eqnarray*}
Since for any representation $V$ of $W'_k$,
$$\epsilon(V)\cdot \epsilon(V^*)= \det V(-1),$$
and as for $V=  \mu \otimes {\rm Ind}_K^k(\chi^{-1})$,
$V^* \cong \mu\det \sigma_1  \otimes {\rm Ind}_K^k(\chi^{-1})$,
it follows that,

\begin{eqnarray*}
\epsilon(\mu \otimes {\rm Ind}_K^k(\chi^{-1})) \cdot
  \epsilon(\mu \det \sigma_1  \otimes {\rm Ind}_K^k(\chi^{-1}))
& = & \det(\mu \otimes {\rm Ind}_K^k(\chi^{-1}))(-1) \\
& = & \omega_{K/k}(-1)\chi(-1).
\end{eqnarray*}
Therefore,
$$\epsilon(\sigma_\pi \otimes {\rm Ind}_K^k(\chi^{-1})) =
\epsilon(\sigma_1 \otimes {\rm Ind}_K^k(\chi^{-1})) \cdot
\omega_{K/k}(-1)\chi(-1).$$

Therefore it follows from the theorem of Saito and Tunnell that
$$\epsilon(\sigma_\pi \otimes {\rm Ind}_K^k(\chi^{-1})) =  1$$
if and only if the character $\chi$ of $K^*$ appears in the
representation $\pi_1$ of $\GL_2(k)$, which by proposition of the
last section are exactly the characters of $K^*$ for which $\pi$ has
Bessel models.

Furthermore, if $\epsilon(\sigma_\pi \otimes {\rm
Ind}_K^k(\chi^{-1})) = - 1$, the representation $\pi_1$ of
$\GL_2(k)$ is a discrete series representation, and if $\pi_1'$ is
the corresponding representation of $D^*$, then $\chi$ appears in
the representation $\pi_1'$ restricted to $K^*$. By proposition
\ref{prop2.2}, the corresponding principal series representation of
$\GSp_D(4)$ has Bessel model for $\chi$, proving Conjecture
\ref{conj1} in this case.

\subsection{Klingen parabolic}

Let $P=MN$ be a Klingen parabolic with $M \cong k^* \times
\GL_2(k)$. Let $\mu \boxtimes \pi_1$ be an irreducible representation of $M$,
giving rise to the principal series representation $\pi$ of
$\GSp_4(k)$ by parabolic induction. Then the Langlands parameter of
the representation $\pi$ of $\GSp_4(k)$ is a representation
$$\sigma_{\pi}: W'_k\rightarrow \GSp_4({\Bbb C}).$$
 Assuming that the Langlands parameter of the
representation $\pi_1$ of $\GL_2(k)$ is $\sigma_1$, and that of
$\mu$ is $\mu$ itself, we have
$$\sigma_\pi = \sigma_1 \oplus \mu \cdot \sigma_1.$$
(This time twisting by $\chi:k^*\rightarrow {\Bbb C}^*$ takes
$(\mu,\pi_1)$ to $(\mu,\chi\pi_1)$.) The central character of $\pi$
is $\mu\det \sigma_1$, therefore the characters $\chi$ of $K^*$
appearing in Conjecture \ref{conj1} have the property that
$\chi|_{k^*}= \mu\cdot \det \sigma_1$, and these are the only
characters that we will consider in what follows.

By standard properties of the local root numbers, for $\sigma_\pi$
as above,

$$\epsilon(\sigma_\pi \otimes {\rm Ind}_K^k(\chi^{-1})) =
\epsilon(\sigma_1 \otimes {\rm Ind}_K^k(\chi^{-1})) \cdot
\epsilon(\sigma_1 \otimes \mu \otimes {\rm Ind}_K^k(\chi^{-1})).$$

Since $\chi|_{k^*} = \mu\cdot \det \sigma_1,$
for $V= \sigma_1 \otimes {\rm Ind}_K^k(\chi^{-1})$, we have
 $V^\vee \cong \sigma_1 \otimes \mu \otimes {\rm Ind}_K^k(\chi^{-1}),$
and therefore,
\begin{eqnarray*}
\epsilon(\sigma_\pi \otimes {\rm Ind}_K^k(\chi^{-1})) & = &
\epsilon(\sigma_1 \otimes {\rm Ind}_K^k(\chi^{-1})) \cdot
\epsilon(\sigma_1 \otimes \mu \otimes {\rm Ind}_K^k(\chi^{-1})) \\
& = & \det (\sigma_1 \otimes {\rm Ind}_K^k(\chi^{-1}))(-1) \\
& = &  [\det (\sigma_1)^2 \cdot \det({\rm Ind}_K^k(\chi^{-1}))^2](-1) \\
&= & 1.
\end{eqnarray*}

Therefore in this case conjecture \ref{conj1} predicts that  $\pi$
has Bessel model for all characters $\chi$ of $K^*$ with
$\chi|_{k^*}= \mu \det \sigma_1$, and this is what proposition 3.3
proves. (Note that the Klingen parabolic is not defined  for the
rank 1 form of $\GSp_4(k)$, so we do not need to consider
$\GSp_D(4)$ here unlike in the case of principal series arising from
Siegel parabolic.)

\section{Reducible principal series}

\subsection{Siegel Parabolic}\label{section5.1}
Let $P=MN$ be a Siegel parabolic with $M \cong \GL_2(k) \times {\Bbb
G}_m $. Let $\pi_1 \boxtimes \mu$ be an irreducible representation of $M$,
giving rise to the principal series representation $\pi=
Ps(\pi_1,\mu)$ of $\GSp_4(k)$ by parabolic induction.

Assume that the Langlands parameter of the representation $\pi_1$ of
$\GL_2(k)$ is $\sigma_1$, and that of $\mu$ is $\mu$ itself. By the
work of Shahidi, it is known that in case $\pi_1$ is supercuspidal,
the principal series representation $\pi= Ps(\pi_1,\mu)$ is
reducible if and only if $\det \sigma_1 =
 |\cdot|^{\pm 1}$. In this case, $\pi= Ps(\pi_1,\mu)$  has two composition
factors, and the Langlands parameter of the generic component is
 $$\sigma = \mu\sigma_1 \oplus (\mu \sqrt{\det\sigma_1}\otimes {\rm St}_2),$$
where ${\rm St}_n$ denotes the $n$-dimensional irreducible
representation of $\SL_2({\Bbb C})$, and that of the other component
is,
 $$\sigma' = \mu\sigma_1 \oplus \mu \cdot{(1\oplus \det\sigma_1)};$$
of course we are assuming that $\det \sigma_1 = |.|^{\pm 1}$,
so $\sqrt{\det \sigma_1} = |.|^{\pm 1/2}$.

Since Conjecture \ref{conj1} is invariant under twisting by a character
of $\GSp_4(k)$, which is identified to a character of $k^*$ via the similitude
character $\GSp_4(k) \rightarrow k^*$, 
we assume by choosing $\mu$ appropriately that the
Langlands parameter of $\pi$ is
 $$\sigma = \sigma_1 \oplus  {\rm St}_2,$$
with $\sigma_1$ of trivial determinant.

We need to calculate the epsilon factor
$$\epsilon(\sigma \otimes {\rm Ind}_K^k(\chi^{-1})) =
\epsilon(\sigma_1 \otimes {\rm Ind}_K^k(\chi^{-1}))
\epsilon({\rm St}_2 \otimes {\rm Ind}_K^k(\chi^{-1})) $$
which takes the value $\pm 1$, and by the theorem of
Saito and Tunnell takes the value $-1$ if and only if
the character $\chi$ of $K^*$ appears in exactly one of the
two representations $\pi_1, {\rm St}_2$.

By the calculation of the Bessel model for a principal series representation,
we already know that the principal series representation has Bessel models
exactly for those characters $\chi$ of $K^*$ which appear in the
restriction of the representation $\pi_1$ to $K^*$.

Thus either of the following two statements implies the other:

\begin{enumerate}
\item The generic component of the principal series representation
has exactly those characters of $K^*$ which appear in $\pi_1$ and in the
Steinberg representation ${\rm St}_2$.

\item The other component of the principal series representation
has exactly those characters of $K^*$ (trivial on $k^*$) which
appear in $\pi_1$ but not in the Steinberg representation ${\rm
St}_2$. Recalling that the Steinberg representation of $\PGL_2(k)$
contains all characters of $K^*$ (trivial on $k^*$) except for the
trivial character, we conclude that the other component of the
principal series has only the trivial character of $K^*$ in the
Bessel model if the trivial character of $K^*$ appears in $\pi_1$,
and that the Bessel model (of the small component of the reducible
principal series) is zero if the trivial character of $K^*$ does not
appear in $\pi_1$.

\end{enumerate}

The method of theta correspondence proves second of the above assertions, so
we are done in this case.

\subsection{Klingen parabolic}\label{section5.2}

Let $P=MN$ be a Klingen parabolic with $M \cong k^* \times
\GL_2(k)$. Let $\mu \boxtimes \pi_1$ be an irreducible representation of $M$,
giving rise to the principal series representation $\pi$ of
$\GSp_4(k)$ by parabolic induction.
 Assume that the Langlands parameter of the
representation $\pi_1$ of $\GL_2(k)$ is $\sigma_1$, and that of
$\mu$ is $\mu$ itself. It is known that if $\pi_1$ is reducible, the
principal series representation $\pi$ is reducible if and only if
$\mu = \omega |\cdot|^{\pm 1}$ for a nontrivial quadratic character
$\omega$ of $k^*$. Assume without loss of generality that $\mu =
\omega |\cdot|$. When reducible, one of the components is a discrete
series representation, with parameter $\sigma_\pi$ which is
$$\sigma_\pi = |\cdot|^{\frac{1}{2}}\sigma_1 \otimes {\rm St}_2, $$
and that of the smaller representation is
$$\sigma_\pi = \sigma_1 \oplus |\cdot| \cdot \sigma_1.$$

Once again since Conjecture 1 is invariant under twisting, we assume
that the reducible principal series is such that the corresponding
generic representation has parameter
$$\sigma_\pi = \sigma_1 \otimes {\rm St}_2. $$
We now need to calculate the epsilon factor,
$$\epsilon(\sigma_\pi \otimes {\rm Ind}_K^k(\chi^{-1})),$$
for this choice of $\sigma_\pi$.

By generalities about epsilon factors,
$$\epsilon(V \otimes {\rm St}_n) = \epsilon(V)^n \det (-F,V^I)^{n-1},$$
where $V^I$ denotes the subspace of $V$ invariant under $I$. In our case,
this formula gives
\begin{eqnarray*}
\epsilon(\sigma_\pi \otimes {\rm Ind}_K^k(\chi^{-1}))
& = & \epsilon(\sigma_1 \otimes {\rm Ind}_K^k(\chi^{-1}) \otimes  {\rm St}_2) \\
& = & \det (-F, V^{I})
\end{eqnarray*}
where $V = \sigma_1 \otimes {\rm Ind}_K^k(\chi^{-1})$. If $V^{I} \neq 0,$ as
$V$ is self-dual, so is $V^I$ as a representation space for the cyclic group
$<F>$. Write $V^I = \sum \chi_i.$ The characters $\chi_i$ with $\chi_i^2 \neq
1$ do not contribute anything to $\det (-F, V^{I})$. A character $\chi_i$
with $\chi_i^2 = 1$, but $\chi_i \neq 1$ also does not contribute to
$\det (-F, V^{I})$. Therefore $\det (-F, V^{I})= (-1)^r$ where $r$ is the
number of copies of the trivial representation of $W_k$ in $V$.
Assuming that
$\sigma_1$ is irreducible, it follows that
$\epsilon(\sigma_\pi \otimes {\rm Ind}_K^k(\chi^{-1})) = -1$ if and
only if $\sigma_1$ and  ${\rm Ind}_K^k(\chi^{-1})$ are isomorphic.

By the calculation of the Bessel model for a principal series representation
done in the last section,
we already know that the principal series representation has Bessel models
for all characters $\chi$ of $K^*$ (with the usual restriction on the
central character).

Thus either of the following two statements implies the other:

\begin{enumerate}
\item The generic component of the principal series representation
$\pi = Ps(\omega |\cdot |, \pi_1)$
has all characters of $K^*$
whose restriction to $k^*$ is the central character of $\pi$ except $\chi$ and
$\chi^{-1}$ if $\sigma_1 = {\rm Ind}_K^k(\chi)$.

\item The other component of the principal series representation
has exactly  characters  $\chi$ and $\chi^{-1}$ of $K^*$
if $\sigma_1 = {\rm Ind}_K^k(\chi)$.

\end{enumerate}

Unfortunately, as neither of the two representations appear as theta lift from
$GO(4)$, we are unable to prove either of the above assertions.

\section{The Steinberg representation}

Let $B$ denote the standard minimal parabolic of $\GSp_4(k)$, and
$P$, $Q$ respectively Siegel and Klingen parabolic subgroups. Let
${\rm St}_4$ denote the Steinberg representation of $\GSp_4(k)$, as
well as its Langlands paramter, which is the four dimensional
representation of the $\SL_2({\Bbb C})$ part of the Weil-Deligne
group $W'_k$. By construction, ${\rm St}_4$ is the alternating sum
of certain representations induced from characters of $B, P,Q,
\GSp_4(k)$ to $\GSp_4(k)$. Ignoring the trivial representation which
does not contribute to the $\psi$-Bessel models for nondegenerate
characters, Steinberg representation can be realized as the quotient
of a representation induced from an irreducible representation of
say $P=MN$  which is twist of the Steinberg representation   of $M$
by a representation of $\GSp_4(k)$ which is induced from a character
of $Q$.

\begin{prop}
Let $K$ be a quadratic algebra, $\chi$ a character of $K^*$ which is
trivial on $k^*$. 
Let $\psi$ be a nondegenerate character
of $N$, left invariant by $K^*$ sitting inside $M$. Then the
Steinberg representation of $\GSp_4(k)$ has Bessel model for $\chi$
if and only if
 $\chi$
is a non-trivial character of $K^*$, in case $K$ is a field, and
for all  characters of $K^*$ if $K = k \oplus k$.
\end{prop}

\begin{proof}

The  proposition is clear by combining
 propositions \ref{prop2.1} and \ref{prop2.4} in all cases except 
when $K=k \oplus k$, and $\chi$ is the trivial character. So the
rest of the proof will be for this case only, which is subtler
as it relies on understanding on semi-simplicity of $\chi$-Bessel
models.

The Steinberg representation sits in an
exact sequence of the form,
$$0 \rightarrow {\rm St}_4 \rightarrow Ps \rightarrow \pi \rightarrow 0$$
where $Ps$ is a principal series representation of $\GSp_4(k)$ induced
from the Siegel parabolic from 
an appropriate twist of the Steinberg representation of 
$M= \GL_2(k) \times k^*$, and $\pi$ is a representation of $\GSp_4(k)$
induced from a one dimensional representation of the Klingen parabolic (in fact
$\pi$ is a subrepresentation of such a representation with quotient
which is the one dimensional trivial representation of $\GSp_4(k)$, so does
not affect the calculation of Bessel models). 
From an earlier observation that the discriminant $\delta_P$ is trivial
on $K^*$, it does not matter which twist of the Steinberg of $\GL_2(k)$ is
used to construct the principal series $Ps$ on $\GSp_4(k)$ above.

Taking the twisted Jacquet functor with respect to the character $\psi$
of $N$, we get an exact sequence of $T$-modules where $T$ is the split
torus in $\GL_2(k)$,
$$0 \rightarrow {\rm St}_{4,\psi} 
\rightarrow Ps_\psi \rightarrow \pi_\psi \rightarrow 0.$$

From the calculation of the twisted Jacquet functor of principal
series done earlier, we get an exact sequence of $T$-modules,

$$0 \rightarrow {\rm St}_{4,\psi} 
\rightarrow {\rm St_2} \rightarrow {\Bbb C} \rightarrow 0$$
where ${\Bbb C}$ is the one dimensional trivial representation of $T$,
and ${\rm St}_2$ denotes the Steinberg representation of $PGL_2(k)$, now
thought of as a $T$-module.

Since the Steinberg representation of $\GL_2(k)$ can be  realized on the
space of locally constant functions modulo constant functions on 
${\Bbb P}^1(k)$, the following is an exact sequence to $T$-modules:

$$0 \rightarrow {\mathcal S}(k^*) \rightarrow {\rm St}_2 \rightarrow 
{\Bbb C} \rightarrow 0.$$

As one knows that the Steinberg representation of $\GL_2(k)$ has a unique
quotient on which $T$-operates trivially, the last two exact sequences
of $T$-modules must be the same, and therefore in particular
${\rm St}_{4,\psi}$ is isomorphic as a $T$-module to ${\mathcal S}(k^*)$.
This implies that 
${\rm St}_4$ has Bessel model for all characters of $T$
which are trivial on the scalars.
\end{proof}

We omit the simple
check that this proposition proves Conjecture \ref{conj1} for the
Steinberg representation, augmented by the following much simpler
proposition.

\begin{prop}
The Steinberg representation of $\GSp_D(4)$ has Bessel model for a
character $\chi$ of $K^*$ if and only if $K$ is a field, and $\chi$
is the trivial character of $K^*$.
\end{prop}

\section{Summing up discussion on principal series}

Since the Gross-Prasad conjecture is about generic $L$-packets, we
need to isolate those irreducible representations of $\GSp_4(k)$
which arise as subquotients of principal series representations, and
have a generic member in their $L$-packet.

We begin by observing that from what is called the {\it Standard
modules conjecture}, which is a theorem for $\GSp_4(k)$, a generic
representation cannot be a proper Langlands quotient, i.e., either
it is already tempered, or it is a full induced representation.

If the representation is tempered, then the sum of the
representations in its  $L$-packet is obtained by inducing a unitary
discrete series representation of a parabolic subgroup of
$\GSp_4(k)$. Thus methods of the previous sections are enough to
handle all representations which are either generic, or are tempered
but not discrete series representations.

Among discrete series representations of $\GSp_4(k)$ which arise as
subquotients of principal series representations, there are the
following possibilities.

\begin{enumerate}
\item Those which arise from the
principal series $Ps(\tau |\cdot |^{1/2}, |\cdot |^{-1/2})$ for the
Siegel parabolic where $\tau$ is a discrete series representation of
$\PGL(2)$ which is not Steinberg. Its $L$-packet has size 2, the
other member being a supercuspidal representation obtained as the
theta lift from $\GSO_4(k)= [D^* \times D^*]/\Delta k^*$ of
the representation $(\lambda, 1)$ of $D^* \times D^*$ where
$\lambda$ is the representation of $D^*$ associated by the
Jacquet-Langlands correspondence to the representation $\tau$ of
$\PGL(2)$. The theta lift is supercuspidal, and the methods of Weil
representation are enough to treat this supercuspidal
representation. The  principal series $Ps(\tau |\cdot |^{1/2},
|\cdot |^{-1/2})$ has another irreducible component which is not
generic, and which is obtained as the theta lift from $\GSO(2,2) =
[\GL_2(k) \times \GL_2(k)]/\Delta k^*$ of the representation
$(\tau,1)$ of $\GL_2(k) \times \GL_2(k)$, and therefore the methods
of the Weil representations give information about the Bessel models
of this other piece of the principal series, and therefore also of
the discrete series component.

\item  Those which arise from the
principal series $Ps(\omega |\cdot|, \pi)$ for the Klingen parabolic
where $\pi$ is a discrete series representation of $\GL(2)$ with
$\omega^2=1$, $\omega \not = 1$, and $\omega \otimes \pi = \pi$. Its
$L$-packet has size one. As mentioned already, we have not been able
to handle these discrete series representations.

\item Twisted Steinberg; these were handled in the previous section.

\end{enumerate}

Finally we note that for $\GSp_4(k)$, a generic $L$-packet
containing a non-generic representation must be tempered, so we have
now handled all (except 2.) generic $L$-packets containing a
representation which is nonsupercuspidal.

\begin{rem} We have so far not discussed in any detail
the unitary principal series
representation of $GSp_4(k)$ which arises from the representation
$1 \boxtimes \pi$
of $k^* \times GL_2(k)$, the Levi subgroup of the Klingen parabolic,
 with $\pi$ a discrete series representation of $GL_2(k)$. The
induced representation in this case is a direct sum of two
irreducible representations, each with Langlands parameter $\sigma
\oplus \sigma$ where $\sigma$ is the Langlands parameter of $\pi$.
It is known that these two components of the principal series
representation can be obtained as theta lift from $GSO(V^a) =[D^*
\times D^*]/\Delta k^*$ of the representation $\pi' \boxtimes \pi'$,
where $\pi'$ is the Jacquet-Langlands lift of $\pi$ to $D^*$, and
the theta lift from $GSO(V^s) =[GL_2(k) \times GL_2(k)]/\Delta k^*$
of the representation $\pi \boxtimes \pi$. Therefore methods of Weil
representation developed later give finer information about the
Bessel periods of these two components of the principal series
representations exactly as predicted by Conjecture \ref{conj1}.
\end{rem}

\section{Bessel model of the Weil representation}

The aim of this section will be to calculate the twisted Jacquet functor of the
Weil representation of a dual reductive pair $(G_1,G_2)$ with respect to a
character $\psi$ of the unipotent radical $N_2$ of a maximal parabolic $P_2=M_2N_2$
of the group $G_2$. We carry out the calculation of the twisted Jacquet functor
 only for the Siegel parabolic of a symplectic group, so $G_2=\Sp(W)$.
Recall that for any representation $\pi$ of $G_2$, the
twisted Jacquet functor $\pi_\psi$ is the maximal quotient of $\pi$ on which
$N_2$ operates via $\psi$. If $M_\psi$ denotes the maximal subgroup of $M_2$
which takes $\psi$ to itself under the inner conjugation action of $M_2$ on $N_2$, then
$\pi_\psi$ is a module for $M_\psi$, and therefore in the context of a dual
reductive pair $(G_1,G_2)$, for $G_1 \times M_\psi$.

We recall that in a famous work, Kudla calculated 
 the standard Jacquet module of the Weil representation.
We carry out the calculation of the twisted Jacquet functor
 only for the Siegel parabolic.
Actually the simple calculation we perform in this section is known
in the literature in both the local and global contexts, see e.g.
\cite{Roberts1, Ra}. However, since we anyway will have to recall
the notation and the results, we have preferred to give an
independent co-ordinate free treatment which will be convenient for
our purposes.

We now recall some elementary properties of the Weil representation for this purpose.

Let $W = W_1 \oplus W_1^\vee$ be a symplectic vector space over a
local field $k$ together with its natural symplectic pairing. Given
a quadratic space $q:V \rightarrow k$, the Weil representation gives
rise to a representation of $\O(V) \times \Sp(W)$ on ${\mathcal S}(V
\otimes W_1^\vee)$. In this representation, elements of $S{\rm
Hom}(W_1^\vee,W_1) = \{\phi \in {\rm Hom}(W_1^\vee,W_1)| \phi =
\phi^\vee\} \cong {\rm Sym}^2W_1,$ which can be identified to the
unipotent radical of the standard Siegel parabolic in $Sp(W)$,
operates on ${\mathcal S}(V\otimes W_1^\vee)$ by
$$(n\cdot f)(x) = \psi((q \otimes q_n)x)f(x),$$
where $n$ is represented by  $q_n \in {\rm Sym}^2W_1$ which gives
rise to a quadratic form $q_n: W_1^\vee \rightarrow k,$ which together
with the  quadratic form  $q:V \rightarrow k$, gives rise to the quadratic
form $q\otimes q_n:V \otimes W_1^\vee \rightarrow k$.

\

The Weil representation associated to a dual reductive pair $(\O(V),
\Sp(W))$ is actually a representation of $\G[\O(V) \times \Sp(W)]$
where $\G[\O(V) \times \Sp(W)]$ is defined to be the group of pairs
$(g_1,g_2) \in \GO(V) \times \GSp(W)$ such that the similitude
factors for $g_1$ and $g_2$ are the same. We briefly recall this,
referring to \cite{Harris-Kudla} for  details on this.

The exact sequence
$$1 \rightarrow \Sp(W)  \rightarrow  \G[\O(V) \times \Sp(W)] \rightarrow \GO(V)
\rightarrow 1,$$ has a natural splitting  $\GO(V) \rightarrow
\G[\O(V) \times \Sp(W)] $ depending on a complete polarization $W =
W_1 \oplus W_2$ of the symplectic space $W$  in which $g \in \GO(V)$
goes to $(g,\mu(g)) \in  \G[\O(V) \times \Sp(W)]  $ where $\mu(g)$
is the element  of  $\GSp(W)$ which acts by 1 on $W_1$ and by
$\nu(g)$ on $W_2$ where $\nu(g)$ is the similitude factor of $g$.
The Weil representation realized on $\mathcal{S}(V \otimes
W_1^\vee)$ has the natural action of $\GO(V)$ operating as
\begin{equation*}
L(h)\varphi(x)= |\nu(h)|^{-mn/2}\varphi(h^{-1}x),
\end{equation*}
where $m = \dim V,$ $2n = \dim W,$ and $\nu(h)$ is the similitude
factor of $h$. The group $\GL(W_1)$ sits naturally inside $Sp(W)$,
and its action on $\mathcal{S}(V\otimes W_1^\vee)$ is given by
\begin{equation*}
L(g)\varphi(x)= \chi_V(\det g)|\det g|^{m/2}\varphi(gx),
\end{equation*}
where $\chi_V$ is the quadratic character of $k^*$ given in terms of
the Hilbert symbol as $\chi_V(a)=(a, {\rm disc} V)$ with ${\rm disc}
V$ the normalised discriminant of $V$. These actions  together with
the action of the Weyl group element (which acts on
$\GL(W_1)$ sitting inside $\Sp(W)$ through $A\rightarrow
{}^tA^{-1}$) of $Sp(W)$ through Fourier
transforms on $\mathcal{S}(V\otimes W_1^\vee)$,
gives the action of $ \G[\O(V) \times \Sp(W)]$ on
$\mathcal{S}(V \otimes W_1^\vee)$. But this analysis of the action
of the Weyl group will not be needed anywhere in this work.

The Weil representation thus gives rise to a representation of
$\G[\O(V) \times \Sp(W)]$; inducing this representation to $\GO(V)
\times \GSp(W)$, we get, the `big Weil representation', say
$\Omega$, of $\GO(V) \times \GSp(W)$. Given an irreducible
representation $\pi$ of $\GO(V)$, there exists a representation
$\Theta(\pi)$ of $\GSp(W)$ of finite length, such that $\Theta(\pi)
\otimes \pi$ is the maximal $\pi$-isotypic quotient of $\Omega$. It
is known that the representation $\Theta(\pi)$ of $\GSp(W)$ has a
unique irreducible quotient $\theta(\pi)$. When one talks about the
theta correspondence, one means the correspondence $\pi \rightarrow
\theta(\pi)$; however, when one calculates Jacquet or twisted
Jacquet functor of the Weil representation, it is invariably
$\Theta(\pi)$ that one encounters. Thus most of the applications are
restricted to the case when one can in fact prove that $\Theta(\pi)
= \theta(\pi)$ which is the case for example when $\pi$ is
supercuspidal.

\begin{lem}\label{lemma2} Let $x$ be a vector in $V \otimes W^\vee_1$, considered as a homomorphism
$x: W_1 \rightarrow V$, as well as the homomorphism on duals $x^\vee: V^\vee \rightarrow W^\vee_1$. Then
for quadratic spaces $q_V: V \rightarrow k$, and $ q_W: W^\vee_1\rightarrow k$, equivalently considered through
 homomorphisms $q_V: V \rightarrow V^\vee$, and $q_W: W^\vee_1 \rightarrow W_1$,
the trace of the map from $W_1$ to $W_1$ given as the compositum of maps,
$$W_1 \overset{x}{\rightarrow} V \overset{q_V}{\rightarrow} V^\vee \overset {x^\vee} {\rightarrow} W^\vee_1
\overset{q_W}{\rightarrow} W_1,$$
is the same as the value of the quadratic form $q_V \otimes q_W$ on the vector $x \in V \otimes W^\vee_1$,
which is of course the same as the trace of the map from $k$ to $ k$, obtained as the compositum of maps:
$$k \overset{x}{\rightarrow} V \otimes W^\vee_1 \overset{q_V \otimes q_W}{\longrightarrow} V^\vee\otimes W_1
 \overset {x^\vee} {\rightarrow} k.$$
\end{lem}

Note that  one can identify the dual of the $k$-vector space
 $S\Hom[W^\vee_1, W_1]$ of self-dual homomorphisms from $W^\vee_1$ to $W_1$
to $S\Hom[W_1, W^\vee_1]$ via the natural pairing obtained by taking trace,
$$ S\Hom[W_1,W^\vee_1] \times S\Hom[W^\vee_1, W_1] \rightarrow \Hom [W_1,W_1] \overset{tr} {\rightarrow } k.$$
Thus characters $\psi: N_2 \rightarrow {\mathbb C}^*$ can be identified to symmetric elements
in $\Hom[W_1,W^\vee_1]$. As the map from $W_1$ to $W_1$ in the above lemma is the compositum
of two maps, one from $W_1$ to $W^\vee_1$, and the other from $W^\vee_1$ to $W_1$, and that the first
map is nothing but the restriction of the quadratic form on $V$ to $W_1$ via the map $x:W_1 \rightarrow V$,
the following corollary of the previous lemma is clear.

\begin{cor}\label{corollary1} The twisted Jacquet functors of the Weil representation corresponding to the
dual reductive pair $(\O(V), \Sp(W))$ are nonzero exactly for the
characters of the unipotent radical of the Siegel parabolic of
$\Sp(W)$ which correspond to the `restriction' of quadratic form on
$V$ to $W_1$ via a linear map $x: W_1 \rightarrow V$.
\end{cor}

We now note the following general lemma.

\begin{lem}\label{lemma3}
Let $X$ be the $k$-rational points of an algebraic variety defined
over a local field $k$. Let $P$ be a locally compact totally
disconnected group with $P= MN$ for a normal subgroup $N$ of $P$
which we assume is a union of compact subgroups. Assume that $P$
operates smoothly on ${\mathcal S}(X)$, and that the action of $P$
restricted to $M$ is given by an action of $M$ on $X$. For a point
$x \in X$, let $\ell_x: {\mathcal S}(X) \rightarrow {\mathbb C}$ be
the linear functional given by $\ell_x(f) = f(x)$. Assume that for
every point $x \in X$, $N$ operates on $\ell_x$ by a character
$\psi_x: N \rightarrow {\Bbb C}^*$, i.e., $\ell_x(n\cdot f) =
\psi_x(n)\ell_x(f)$ for all $n \in N$, and $f \in {\mathcal S}(X)$.
Fix a character $\psi: N \rightarrow {\mathbb C}^*$, and let
$M_\psi$ denote the subgroup of $M$ which stabilizes the character
$\psi$ of $N$. The group $M_\psi$ acts on the set of points $x \in
X$ such that $\psi_x = \psi$. Denote this set of points in $X$ by
$X_\psi$ which we assume to be closed in $X$. Then,
$${\mathcal S}(X)_\psi \cong {\mathcal S}(X_\psi)$$
as $M_\psi$-modules.
\end{lem}
\begin{proof} We have an exact sequence of $M_\psi$-modules,
$$0 \rightarrow {\mathcal S}(X-X_\psi) \rightarrow {\mathcal S}(X) \rightarrow {\mathcal S}(X_\psi) \rightarrow 0.$$
Taking the $\psi$-twisted Jacquet functor is exact, and ${\mathcal
S}(X-X_\psi)_\psi = 0$, so the assertion of the lemma follows.
\end{proof}

\begin{cor}\label{corollary2}
The twisted Jacquet functor of the Weil representation of the dual
reductive pair $(\O(V), \Sp(W))$  for the character of the unipotent
radical of the Siegel parabolic of $\Sp(W)$ which corresponds to a
nondegenerate quadratic form on $W_1$, which we assume is obtained
by restriction of the quadratic form on $V$  via a linear map $x:
W_1 \rightarrow V$ is the representation
$${\rm ind}^{\O(V) \times \O(W_1)} _{ \O(W^\perp_1) \times \Delta \O(W_1)} {\mathbb C},$$
where $O(W^\perp _1)$ is the orthogonal group of the orthogonal
complement of $W_1$ inside $V$, and $\Delta \O(W_1)$ represents the
natural diagonal embedding of $\O(W_1)$ inside $\O(V) \times
\O(W_1)$ as $V$ contains $W_1$.
\end{cor}

\begin{proof} Observe that $\O(V) \times \O(W_1)$ operates on the set of
homomorphisms from $W_1$ to $V$, and in fact by Witt's theorem, this
action is transitive on the set of homomorphisms from $W_1$ to $V$
such that the quadratic form on $V$ restricts to the quadratic form
on $W_1$. The isotropy subgroup inside $\O(V) \times \O(W_1)$ of a
fixed embedding of $W_1$ inside $V$ is exactly $\O(W^\perp_1) \times
\Delta \O(W_1)$, proving the claim.
\end{proof}

The previous analysis of twisted Jacquet functor in fact gives a
representation space for $\G[\O(V) \times \O(W_1)]$ which we record
as the following corollary, but before doing that let us record the
following form of Witt's extension theorem for similitude groups
which follows from the usual form of Witt's theorem.

\begin{lem}\label{lemma4}
Suppose $W_1$ is a subspace of a quadratic  space $V$ carrying the
restricted quadratic form. Suppose $\phi$ belongs to $\GO(W_1)$ such
that the similitude factor of $\phi$ arises as a similitude factor
in $\GO(V)$. Then there is an element $\phi'$ in $\GO(V)$ taking
$W_1$ into itself, and such that the restriction of $\phi'$ to $W_1$
is $\phi$.
\end{lem}

\begin{cor}\label{corollary3}
The twisted Jacquet functor of the Weil representation of the dual
reductive pair $(\O(V), \Sp(W))$  for the character of the unipotent
radical of the Siegel parabolic of $\Sp(W)$ which corresponds to a
nondegenerate quadratic form on $W_1$, which we assume is obtained
by restriction of the quadratic form on $V$  via a linear map $x:
W_1 \rightarrow V$ is the representation
$${\rm ind}^{\G[\O(V) \times \O(W_1)]} _{ \G[\O(W^\perp_1) \times \Delta \O(W_1)]} {\mathbb C},$$
where $\O(W^\perp _1)$ is the orthogonal group of the orthogonal
complement
 of $W_1$ inside $V$, and
$\Delta \O(W_1)$ represents the natural diagonal embedding of
$\O(W_1)$ inside $\O(V) \times \O(W_1)$ as $V$ contains $W_1$; the
group ${ \G[\O(W^\perp_1) \times \Delta \O(W_1)]}$ is the subgroup
of ${ \GO(W^\perp_1) \times \Delta \GO(W_1)}$ consisting of the
pairs $(g_1,g_2)$ with the same similitude factors for $g_1$ and
$g_2$.

\end{cor}

As recalled earlier in the section, the Weil representation associated to the dual reductive pair
$(\O(V),\Sp(W))$ can actually be made to be a representation of
$\GO(V) \times \GSp(W)$ simply by inducing the Weil representation
of $\G[\O(V) \times \Sp(W)]$ to the larger group
 $\GO(V) \times \GSp(W)$.
Here, then,  is the extension of the previous corollary to the case of similitude groups.

\begin{cor}\label{corollary4}
The twisted Jacquet functor of the Weil representation of the dual
reductive pair $(\GO(V), \GSp(W))$  for the character of the
unipotent radical of the Siegel parabolic of $\Sp(W)$ which
corresponds to a nondegenerate quadratic form on $W_1$, which we
assume is obtained by restriction of the quadratic form on $V$  via
a linear map $x: W_1 \rightarrow V$ is the representation
$${\rm ind}^{\GO(V) \times \GO(W_1)} _{ \G[\O(W^\perp_1) \times \Delta \O(W_1)]}
{\mathbb C},$$ where $\O(W^\perp _1)$ is the orthogonal group of the
orthogonal complement
 of $W_1$ inside $V$, and
$\Delta \O(W_1)$ represents the natural diagonal embedding of
$\O(W_1)$ inside $\O(V) \times \O(W_1)$ as $V$ contains $W_1$; the
group ${ \G[\O(W^\perp_1) \times \Delta \O(W_1)]}$ is the subgroup
of ${ \GO(W^\perp_1) \times \Delta \GO(W_1)}$ consisting of
$(g_1,g_2)$ with the same similitude factors for $g_1$ and $g_2$.
Assuming that $W_1 ^\perp \neq 0$, the twisted Jacquet functor as a
representation space of $\GSO(V) \times \GSO(W_1)$ is therefore
$${\rm ind}^{\GSO(V) \times \GSO(W_1)} _{ \G[\SO(W^\perp_1) \times
\Delta \SO(W_1)]} {\mathbb C}.$$

\end{cor}

The previous corollary  together with the formalism of the Weil representation yields
the following theorem as a simple consequence.

\begin{thm}\label{thm4} Let $\pi_1$ be an irreducible admissible representation of $\GSO(V)$, and
$\pi_2$ that of $\GSp(W)$. Assume that $\pi_2 = \Theta(\pi_1)$ is
the theta lift of $\pi_1$ to $\GSp(W)$. Let $\psi$ be a
nondegenerate character of the unipotent radical $N$ of the Siegel
parabolic $P=MN$ of $\GSp(W)$. Assume that $\psi$ corresponds to a
quadratic form $q$ on $W_1$, a maximal isotropic subspace of $W$.
Then an irreducible  representation $\chi$ of $\GSO(W_1)$ appears in
$\pi_{2,\psi}$ as a quotient if and only if

\begin{enumerate}
\item $(q,W_1)$ can be embedded in the quadratic space $V$; let $W_1^\perp$
denote the orthogonal complement of $W_1$ sitting inside $V$
through this embedding.

\item The representation $\chi^\vee$ of $\G[\SO(W_1) \times \SO(W_1^\perp)]$
appears as a quotient in the representation $\pi_1$ of $\GSO(V)$
restricted to $\G[\SO(W_1) \times \SO(W_1^\perp)]$, where
$\chi^\vee$ is obtained by pulling back the contragredient of $\chi$
under the natural map $\G[\SO(W_1) \times \SO(W_1^\perp)]
\rightarrow \GSO(W_1) $.
\end{enumerate}
\end{thm}

\begin{rem} It is a consequence of this theorem that if
the representation $\chi^\vee$ of $\G[\SO(W_1) \times
\SO(W_1^\perp)]$ appears as a quotient in the representation $\pi_1$
of $\GSO(V)$ restricted to $G[\SO(W_1) \times \SO(W_1^\perp)]$, then
$\pi_2 = \Theta(\pi_1)$ is nonzero. It is one of the standard ways
by which one proves nonvanishing of local (or global)
representations: by proving the nonvanishing of a particular Fourier
coefficient; for example it proves that the theta lifting from
$\GSO(4)$ to $\GSp(4)$ is always nonzero locally.
\end{rem}
\section{Applications }\label{sect:applications}

To be able to use Theorem \cite{thm4}, we need to understand the embedding  of
$\O(W_1)$ inside $\O(V)$ more concretely. For applications to
Theorem \cite{thm1.4},  we need it  in the case when $V$ is a four dimensional
quadratic space, and $W_1$ is a two dimensional subspace of it, and
for applications to Theorem \cite{thm1.3},  we need it  in the case when $V$
is a 6 dimensional quadratic space, and $W_1$ is a two dimensional
subspace of it.

We begin with the case of a four dimensional quadratic space $V$ of
discriminant 1, so that it can be identified to the norm form of a
four dimensional central simple algebra, say $D$, over $k$. Assume
that the two dimensional nondegenerate
subspace $W_1$ of $V=D$ is the norm form on a two dimensional
subalgebra $K$ of $D$. Write $D = K \oplus K\cdot j$ where $j$ is an
element of $D^*$ which normalises $K^*$ with $j^2 = a \in k^*$. The
group $D^* \times D^*$ operates on $D$ by $(d_1,d_2)X=
d_1X\bar{d}_2$, and gives an identification of $[D^* \times D^*]/\Delta k^*$
with $\GSO(D)$. 
Observe that the map $\iota: (x,y) \rightarrow 
(xy,x\bar{y})$ from $K^* \times K^*$
to itself gives an isomorphism of $(K^* \times K^*)/\Delta k^*$ onto
the subgroup $\G[\SO(W_1) \times \SO(W^\perp_1)]$ of
$\GSO(W_1) \times \GSO(W^\perp_1)$ consisting of pairs of elements
of $K^*$ with the same similitude factors for the two components.
Since $x[K \oplus Kj]\bar{y} = x\bar{y} K \oplus x{y} Kj$, 
the following 
diagram allows one to identify $(K^* \times K^*)/\Delta k^*$ 
inside $(D^* \times D^*)/\Delta k^*$ as the subgroup 
$G[SO(K) \times SO(K)]$ inside
$GSO(D)=GSO(K\oplus K)$

$$
\xymatrix{ &[K^*\times K^*]/ \Delta(k^*) \ar@{->}[dl]_{\cong}
\ar@{->}[dr]^{} \\
\G[\SO(K)\times \SO(K)] & & ( D^* \times D^*)/(\Delta k^*)}
$$

Therefore a representation $\pi_1 \boxtimes \pi_2$ of $D^*
\times D^*$ contains the restriction of the character
$(\chi_1,\chi_2)$ of $K^* \times K^*$ to the subgroup $\G[\SO(W_1)
\times \SO(W^\perp_1)]$ if and only if $\chi_1\bar{\chi}_2$ appears in
$\pi_1$ and $(\chi_1{\chi}_2)$ appears in $\pi_2$.
Taking $\chi_2 =1$, we get the following corollary to Theorem \cite{thm4}.

\begin{cor}\label{corollary5}

Let $\pi_1 \boxtimes \pi_2$ be  an irreducible admissible
representation of $[D^* \times D^*]/k^* = \GSO(V)$ where $V=D$ is a
quaternion algebra over $k$ equipped with its reduced norm as the
quadratic form. Let $\psi$ be a character of the unipotent radical
of the Siegel parabolic of $\GSp(W)$ which corresponds to the
nondegenerate quadratic space  ${\rm Nm}: K \rightarrow {k}$ 
where $K$ is a quadratic subalgebra of $D$.
Then for the
representation  $\Theta(\pi_1 \boxtimes \pi_2)$ of $\GSp(W)$, the
twisted Jacquet functor, $\Theta_\psi(\pi_1 \boxtimes \pi_2)$ of
$\GSp(W)$, contains the representation $\chi: K^* \rightarrow
{\mathbb C}^*$ if and only if    $\chi$ appears in both $\pi_1$ and
 $\pi_2$.

\end{cor}

Similarly, for the case of the rank one form $\Sp_4(D)$ of the
symplectic group defined using the quaternion division algebra $D$,
we get the following result.

\begin{cor}\label{corollary6}

Let $\pi_1 \boxtimes \pi_2$ be a representation of $[D^* \times
\GL_2(k)]/k^* = \GSO_D(4)$ where $D$ is a quaternion algebra over
$k$, Let $\psi$ be a character of the unipotent radical of the
Siegel parabolic of $\GSp(W)$ which corresponds to the nondegenerate
quadratic space  ${\rm Nm}: K \rightarrow {k}$ 
where $K$ is a quadratic subalgebra of $D$.
Then for the representation
$\Theta(\pi_1 \boxtimes \pi_2)$ of $\GSp(W)$, the twisted Jacquet
functor, $\Theta_\psi(\pi_1 \boxtimes \pi_2)$ of $\GSp(W)$, contains
the representation $\chi: K^* \rightarrow {\mathbb C}^*$ if and only
if $\chi$ appears in both $\pi_1$ and  $\pi_2$.

\end{cor}

Let $K$ and $L$ be two quadratic extensions of $k$, and let $E$ be
the third quadratic extension of $k$ contained in $KL$. Considering
$K$ and $L$ together with their norm forms, we have two dimensional
quadratic spaces, and $K \oplus L$ is a four dimensional quadratic
space. It can be seen that $\GSO(K\oplus L) \cong (\GL_2(E) \times
k^*)/\Delta E^*$ where $\Delta E^* \cong E^*$ sits inside $\GL_2(E)$
as scalar matrices, and inside $k^*$ via the norm mapping.

The group $\G[\SO(K) \times \SO(L)] $ is the subgroup of $K^* \times
L^*$ consisting of pairs $(x_1,x_2) \in K^* \times L^*$ with the
same norm to $k^*$.

The mapping from $\G[\SO(K) \times \SO(L)] $
 to $(\GL_2(E) \times k^*)/\Delta E^*$ obtained as the
composition,
$$\G[\SO(K) \times \SO(L)]  \rightarrow \GSO(K\oplus L) \cong (\GL_2(E) \times k^*)/\Delta E^*$$
fits in the following diagram of maps where $\phi_E$ denotes the
natural inclusion of $(KL)^* $ into $\GL_2(E)$, and $i,i_K,i_L$ are
inclusions of $k^*$ in $k^*, K^*,L^*$ respectively, and ${\rm Nm}_K$
and ${\rm Nm}_L$ are norm mappings from $(KL)^*$ to $K^*$ and $L^*$
respectively.

$$
\xymatrix{ &[(KL)^*\times k^*]/ \Delta(E^*) \ar@{->}[dl]_{(i_K{\rm Nm}_K, i_L{\rm Nm}_L)} \ar@{->}[dr]^{(\phi_E, i)} \\
\G[\SO(K)\times \SO(L)] & & ( \GL_2(E) \times k^*)/(\Delta E^*)}
$$
As the arrow on the left can be checked to be an isomorphism,
it follows from this
diagram that to check that a character of $\G[\SO(K) \times \SO(L)]
$ appears in the restriction of a representation of $\GSO(K \oplus
L) $, it is equivalent to check that its restriction to $[(KL)^*
\times k^*]/ \Delta E^*$ now thought of as a subgroup of $[\GL_2(E)
\times k^*]/\Delta(E^*)$ appears in the corresponding representation
of $[\GL_2(E)\times k^*]/\Delta(E^*)$.  Therefore we obtain the
following theorem.

\begin{thm}\label{thm8.1}

Let $\pi_1$ be an irreducible admissible representation of
$\GSp_4(k)$ obtained from the theta lift of a representation $\pi$
of $\GO_4(k)$ such that the normalized discriminant algebra
associated to the four dimensional quadratic space is a quadratic field
extension $E$ of $k$. Assume that in the identification of
$\GSO_4(k)$ with $( \GL_2(E) \times k^*)/(\Delta E^*)$, the
restriction of $\pi$ (from $\GO_4(k)$ to $\GSO_4(k)$) corresponds to
the representation $\pi_2 \boxtimes \mu$ of $\GL_2(E) \times k^*$. Let
$\psi$ be a nondegenerate character of $N$, where $N$ is the
unipotent radical of the Siegel parabolic $P=MN$ stabilizing a
maximal isotropic subspace $W_1$ of the four dimensional symplectic
space $W$,
 corresponding
to a quadratic form $q$ on $W_1$ which defines a quadratic field 
extension $K \not = E$. (The case $K=E$ is easier to analyze but we 
do not do it here.)
Then a character $\chi$ of $K^*$ such that $\chi|_{k^*}$ is the central character
of $\pi_1$, appears in $\pi_{1,\psi}$ if and only if the character $\chi\circ {\rm Nm}: (KL)^*
\rightarrow K^* \rightarrow {\Bbb C}^*$ of $(KL)^*$ appears in the restriction
of $\pi_2$ to $(KL)^*$ which by the theorem of Saito and
Tunnell is the case if and only if
\begin{eqnarray*}
\omega_{KL/E}(-1)\omega_{\pi_2}(-1) & = &
\epsilon(\pi_2 \otimes {\rm ind}_{KL}^{E} \chi^{-1}|_{KL}) \\
& = &  \epsilon(\pi_2 \otimes
{\rm Res}_E[ {\rm ind}_{K}^{k} (\chi^{-1})]) \\
& =  &
\epsilon({\rm ind}_E^k(\pi_2) \otimes {\rm ind}_{K}^{k} (\chi^{-1})).
 \end{eqnarray*}
Noting the generality that $\omega_{KL/E}(-1)=1$, and that
$\omega_{\pi_2}(-1)=1$ as $\pi_2$ is a representation of $\GL_2(E)$
which extends to a representation of $( \GL_2(E) \times k^*)/(\Delta
E^*)$, its central character restricted to $E^1$ is trivial, we get
that
$$\epsilon({\rm ind}_E^k(\pi_2) \otimes {\rm ind}_{K}^{k} (\chi^{-1}))=1$$
if and only if the character $\chi$ appears in the Bessel model of
$\pi$ as required by Conjecture \ref{conj1}.
\end{thm}

\begin{rem} There is a form of this theorem for the rank 1 form
$\GSp_D(4)$ of $\GSp_4(k)$ too in which one would be considering
theta lifting from an orthogonal group in 4 variables defined  using
$D$ by taking a skew-Hermitian matrix in $\GL_2(D)$ whose
discriminant in $k^*/k^{*2}$ defines a quadratic extension $E$ of
$k$. In this case, the orthogonal group turns out to be $(D_E^*
\times k^*)/E^*$ with $D_E$ the unique quaternion division algebra
over $E$, and a similar analysis as done for deducing the previous
theorem works exactly in the same way, confirming Conjecture
\ref{conj1} for such representations of $\GSp_D(4)$.
\end{rem}

\section{Dual pairs involving division algebras}
In this section we briefly recall the formalism of dual reductive
pairs which involve quaternion division algebras; the final goal of this
section will be to state the analogue of theorem 4 in this context.

Let $D$ be a  quaternion division algebras with its canonical involution
$x\rightarrow \bar{x}$. Using this involution, right $D$-modules
can be identified to left $D$-modules.

Let $V$ be a right $D$-module, and $H: V \times V \rightarrow D$
a $\epsilon$-hermitian form on $V$ which is linear in the second
variable, so that
\begin{enumerate}
\item $H(v_1d_1,v_2d_2)= \bar{d}_1H(v_1,v_2)d_2.$

\item $\overline{H(v_1,v_2)} = \epsilon H(v_2,v_1).$ (This forces $\epsilon$ to
be $\pm 1$.)
\end{enumerate}

If $\epsilon = 1$ (resp., $\epsilon = -1$), an $\epsilon$-hermitian form is
called hermitian (resp., skew-hermitian).

Let $V_1$ be a right $D$-module together with a $\epsilon_1$-hermitian
form linear in the second variable, and $V_2$ a left $D$-module
together with a $\epsilon_2$-hermitian form $H_2$ which is linear
in the first variable. Then $V_1 \otimes_DV_2$ is a vector space over
$k$ together with a natural bilinear form $H = H_1 \otimes H_2$ given
by
$$H(v_1\otimes v_2,w_1\otimes w_2) = {\rm tr}_{D/k}({H_1(v_1,w_1)}
\overline{H(v_2,w_2)}).$$ If $\epsilon_1\epsilon_2 = -1$, as will
always be the case in what follows, $H$ will be a symplectic form on
$V_1\otimes_D V_2$. In this case, the isometry group $G_1$ of
$(V_1,H_1)$ to be denoted by $U(V_1)$, and $G_2$ of $(V_2,H_2)$ to
be denoted by $U(V_2)$, form a dual reductive pair inside
$\Sp(V_1\otimes_DV_2)$ in the sense of Howe. We let $\GU(V_1)$ and
$\GU(V_2)$ denote the corresponding similitude groups.

It is known that to get a form of an orthogonal group, we need to
take a skew-hermitian form, and that to get a form of the symplectic
group, we need to take a hermitian form. As an example of interest
for our work, for $a \in D^*$, let $D(a)$ denote the one dimensional
right $D$-module which is $D$ itself together with the form
$H(d_1,d_2)= \bar{d}_1a d_2$. This form is skew-hermitian if  $a +
\bar{a}=0$, and hermitian if $a = \bar{a}$. Assuming $a$ is such
that $a + \bar{a}=0$, the group $U(D(a))$ is an orthogonal group in
two variables, and $GU(D(a)) = K^*$ where $K$ is the quadratic
extension of $k$ generated by $a$.

Assume that $H_1$ is a skew-hermitian form on $V_1$, and $H_2$ is a
hermitian form on $V_2$. Let $V_2=W_2 \oplus W_2^\vee$ be a complete
polarization of $V_2$. The weil representation of
$\Sp(V_1\otimes_DV_2)$ is realized on the Schwartz space of
functions on $V_1 \otimes_D W_2^\vee$ on which $\U(V_1)$ acts in the
natural way. The polarization  $V_2=W_2 \oplus W_2^\vee$ gives rise
to the parabolic $P$ in $\U(V_2)$ stabilizing the subspace $W_2$
with $\GL(W_2)$ as the Levi subgroup, and the additive group of
skew-hermitian forms on $W_2^\vee$ as $N$. Thus the character group
of $N$ can be identified to the additive group of skew-hermitian forms on
$W_2$.

With these preliminaries, we state the analogue of Theorem
\ref{thm4} in this context; application of this result to theta
lifting between
 $\GSp_D(4)$, and $\GSO_D(4)$ will not be explicitly stated.

\begin{thm} Let $\pi_1$ be an irreducible admissible representation of
$\GU(V_1)$, and $\pi_2$ that of $\GU(V_2)$. Assume that $\pi_2 =
\Theta(\pi_1)$ is the theta lift of $\pi_1$ to $\GU(V_2)$. Let
$\psi$ be a nondegenerate character of the unipotent radical $N$ of
the Siegel parabolic $P=MN$ of $\GU(V_2)$ stabilizing a maximal
isotropic subspace $W_2$ of $V_2$. Assume that $\psi$ corresponds to
a skew-hermitian form $H$ on $W_2$. Then an irreducible
representation $\chi$ of $\GU(W_2)$ appears in $\pi_{2,\psi}$ as a
quotient if and only if

\begin{enumerate}
\item $(H,W_2)$ can be embedded in the skew hermitian space $V_1$;
let $W_2^\perp$
denote the orthogonal complement of $W_2$ sitting inside $V_1$
through this embedding.

\item The representation $\chi^\vee$ of $\G[\U(W_2) \times \U(W_2^\perp)]$
appears as a quotient in the representation $\pi_1$ of $\GU(V_1)$
restricted to $\G[\U(W_2) \times \U(W_2^\perp)]$, where $\chi^\vee$
is obtained by pulling back the contragredient of $\chi$ under the
natural map $\G[\U(W_2) \times \U(W_2^\perp)]  \rightarrow \GU(W_2)
$.
\end{enumerate}
\end{thm}

\noindent{\bf Example :} The orthogonal group defined by the skew-hermitian
form 
$$\left ( \begin{array}{cc} a & 0 \\ 0 & b
\end{array} \right ), {\rm~~~~~for~~ a,b} \in D^*, 
{\rm~~tr}(a)={\rm~~tr}(b)=0,$$ defines an orthogonal group in four 
variables which is,
\begin{enumerate}
\item $GSO_D(4)\cong  [D^* \times GL_2(k)]/k^*$ if $ab \in k^*$;
\item $GSO_D(4) \cong [D_E^* \times k^*]/E^*$ if $ab \in E^*$ 
(and not $k^*$), where $E$ 
is a quadratic extension of $k$, and $D_E$ is the unique
quaternion division algebra over $E$.
\end{enumerate}

\section{Proof of Conjecture \ref{conj1} for representations arising from Weil
representation through $\GSO(4)$}

Most of the work to prove Conjecture \ref{conj1} is already done. We
recall the Langlands parameter of representations $\GSp_4(k)$
arising from theta correspondence from representations of $\GO(4)$,
and then do the necessary epsilon factor calculation to verify
Conjecture \ref{conj1} from results proved in the last section.

As recalled in the introduction, for a four dimensional quadratic
space $V$, $\GSO(V)$ has the structure of one of the following
groups:

\begin{enumerate}
\item $\GSO(V^s) \cong (\GL_2(k) \times \GL_2(k) )/\Delta k^*.$

\item  $\GSO(V^a) \cong (D^* \times D^*)/\Delta k^*. $

\item $\GSO(V^d) \cong [\GL_2(E) \times k^*]/\Delta E^*,$
\end{enumerate}
where $\Delta k^* = k^*$ sits as $(t,t^{-1})$, and
$\Delta E^* = E^*$ sits inside $\GL_2(E) \times k^*$ via its natural
embedding in $\GL_2(E)$, and in $k^*$ by the inverse of the norm mapping.

In cases $(1)$ and $(2)$, an irreducible representation of $\GSO(V)$
is a tensor product of irreducible representations $\tau_1 \boxtimes
\tau_2$ where $\tau_1$ and $\tau_2$ are both irreducible
representations of $\GL_2(k)$ or of $D_k^*$, and so have Langlands
parameters $\sigma_1$ and $\sigma_2$. The Langlands parameter of the
representation of  $\GSp(4)$ arising from theta correspondence from
an irreducible
 representation of $\GO(V)$ which restricted
to $\GSO(V)$ is $\tau_1 \boxtimes \tau_2$ in  cases $(1)$ and $(2)$
is,
$$\sigma_1 \oplus \sigma_2.$$

In case $(3)$, an irreducible representation of $\GSO(V^d)$
corresponds to an irreducible representation $\tau$ of $\GL_2(E)$
whose central character is invariant under ${\rm Gal}(E/k)$,
together with a character $\chi$ of $k^*$ such that the central
character of $\tau$ can be considered to be the character of $E^*$
obtained from the character $\chi$ of $k^*$ through the norm
mapping. In this case, the Langlands parameter of the representation
$\GSp_4(k)$ arising from theta correspondence from this
representation of $\GO(V^d)$  is,
$${\rm Ind}_E^k \sigma.$$

The
epsilon factor $\epsilon(\sigma \otimes {\rm Ind}_K^k(\chi^{-1}))$
in cases $(1)$ and $(2)$ is simply the product of the epsilon
factors, $\epsilon(\sigma_1 \otimes {\rm Ind}_K^k(\chi^{-1}))$ and
$\epsilon(\sigma_2 \otimes {\rm Ind}_K^k(\chi^{-1}))$ which by the
theorem of Saito and Tunnell can be easily interpreted in terms of
the existence of the character $\chi$ of $K^*$ in the
representations $\tau_1$, $\tau_2$, making Theorem \ref{thm1.4} a consequence
of Corollaries \ref{corollary5} and \ref{corollary6}. Similarly in
case $(3)$, conjecture 1 is equivalent to Theorem \ref{thm8.1}.

 These deductions have been made assuming of course that $\Theta(\pi)
= \theta(\pi)$, which is true in particular when $\pi$ is a
supercuspidal representation. We use the methods of theta for
exactly these representations, except one more case when the
representation of $\GSO(V^s)$ is $\tau= \tau_1 \boxtimes \tau_2$ with
$\tau_1$ supercuspidal of trivial central character and $\tau_2$ the
trivial representation. In this case it is easy to see that the
theta lift of $\tau$ is a nongeneric representation of $\GSp_4(k)$
which has no choice but to be  the nongeneric component of the
principal series representation of $\GSp_4(k)$ induced from the
representation $(\tau_1 |\cdot |^{1/2}, |\cdot|^{-1/2})$ of the
Siegel parabolic. In this case corollary 6 implies that the only
characters $\chi: K^* \rightarrow {\Bbb C}^*$ for which there is a
Bessel model for $\pi = \theta(\tau) = \Theta(\tau)$ is the trivial
character if it  appears in $\tau_1$ restricted to $K^*$. This is
exactly the conclusion that was desired at the end of Section
\ref{section5.1}, proving Conjecture 1 in this case.

Combining the results for irreducible principal series, reducible
principal series representations induced from supercuspidal
representations of the Siegel parabolic, completely reducible
principal series representations (thus forming a single $L$-packet),
twists of Steinberg, we are left exactly with those reducible
principal series arising out of supercuspidal representations of
Klingen parabolic which have a non-trivial self-twist, and whose
Langlands parameters are of the form $\sigma \otimes {\rm St}_2$ for
a two dimensional representation of the Weil group of $k$ with a
non-trivial self-twist. As these representations do not arise from
$\GO(4)$, we seem to be out of luck dealing with such
representations, for which an explicit suggestion about Bessel model
was made at the end of Section \ref{section5.2}.

\section{Theorem \ref{thm1.3}}

We next consider the case of the dual reductive pair
$(\Sp(4),\O(6))$ where we will assume that $\O(6)$ is either split,
or is a rank 1 form of it. Thus $\GSO(6)$ will be one of the two
groups:

\begin{enumerate}
\item $[\GL_4(k) \times k^*]/\{(z,z^{-2}): z \in k^*\},$

\item $[\GL_2(D) \times k^*]/\{(z,z^{-2}): z \in k^*\}.$

\end{enumerate}

In this case we will be looking at the embedding of a two
dimensional space in a six dimensional space, say $K \hookrightarrow
K \oplus K \oplus {\Bbb H}$, a direct sum of quadratic spaces, which
gives an  embedding of $\G[\SO(K) \times \SO(K \oplus {\Bbb H})]$
inside $\GSO(K \oplus K \oplus {\Bbb H})$. We remind ourselves that
$$\GSO(K \oplus {\Bbb H}) \cong [\GL_2(K) \times k^*]/K^*,$$
where $K^*$ sits inside $ [\GL_2(K) \times k^*]$ as $(x,{\rm Nm}
x)$. Therefore there is a natural embedding of $G[\SO(K) \times
\SO(K \oplus {\Bbb H})]$ inside $K^* \times [\GL_2(K) \times
k^*]/K^* $. We claim that under this embedding, the image of
$\G[\SO(K) \times \SO(K \oplus {\Bbb H})]$ inside $K^*  \times
[\GL_2(K) \times k^*]/K^* $ can be identified to $[\GL_2(K) \times
k^*]/k^* $ where $k^*$ sits naturally as the scalar matrices in
$\GL_2(K)$, and in $k^*$ through $t\rightarrow t^2$. To prove this
claim, note that there is a natural map from  $[\GL_2(K) \times
k^*]/k^* $ to $[\GL_2(K) \times k^*]/K^* $, and therefore to $K^*
\times [\GL_2(K) \times k^*]/K^* $ in which $(X,t) $ goes to
$t^{-1}\det X$ in $K^*$. It is easy to check that this map is
injective, and its image is exactly $\G[\SO(K) \times \SO(K \oplus
{\Bbb H})]$.

Using the identifications indicated above, the embedding of
$\G[\SO(K) \times \SO(K \oplus {\Bbb H})]$ inside $\GSO(K \oplus K
\oplus {\Bbb H})$, becomes the standard embedding of $[\GL_2(K)
\times k^*]/k^*$ inside $[\GL_4(k) \times k^*]/\{(z,z^{-2}): z \in
k^*\},$ or inside $[\GL_2(D) \times k^*]/\{(z,z^{-2}): z \in k^*\}$
as the case may be, and further the natural map from $\G[\SO(K)
\times \SO(K \oplus {\Bbb H})] = [\GL_2(K) \times k^*]/k^*$ to $K^*=
\GSO(K)$ appearing in theorem 4 is nothing but $(X,t) $ goes to
$t^{-1}\det X$ in $K^*$, and thus theorem 4 detects the appearance
of one dimensional representations of $\GL_2(K)$ as a quotient of a
representation of $\GL_4(k)$ which arise from theta lifting from
$\GSp_4(k)$.

From the work of Gan and Takeda \cite{gan-takeda2}, it follows that
a representation of $\GL_4(k)$ arises as a theta lift from
$\GSp_4(k)$ if and only if its Langlands parameter belongs to the
symplectic similitude group $\GSp_4({\Bbb C})$. By the remark
following theorem 4, as soon as a character of $\GL_2(K)$ appears as
a quotient of
 a representation of $\GL_4(k)$, the representation of $\GL_4(k)$ arises
 from theta
lifting from $\GSp_4(k)$, and therefore its parameter belongs to the
symplectic similitude group.

To get the finer assertion in Theorem \ref{thm1.3} regarding the
epsilon factors, one needs to just use  the theorem about Bessel
models for $\GSp_4(k)$, which is what we just proved in the earlier
sections in odd residue characteristic. In even residue
characteristic, since Bessel models have not been completely
determined for `exceptional' representations of $\GSp_4(k)$, same
gap remains here.

We also note that as usual the methods of theta correspondence give
results only for those irreducible representations of $\GL_4(k)$
which arise as $\Theta(\pi)$ for an irreducible representation $\pi$
of $\GSp_4(k)$, therefore once again we will use the methods of
theta correspondence only for supercuspidal representations of
$\GL_4(k)$.  Other representations of $\GL_4(k)$ for which there is
a character of $\GL_2(K)$ appearing in it as a quotient, must arise
from parabolic induction of an irreducible representation of the
$(2,2)$ parabolic (as their parameter is in $\GSp_4({\Bbb C})$).
If we are dealing with non-discrete series but generic
representation of $\GL_4(k)$, we can assume that the representation
is a full induced representation from an irreducible representation
of the $(2,2)$ parabolic.

For the full induced representation from the $(2,2)$ parabolic subgroup,
 the Mackey theory which answers questions about
 restriction of an induced representation to a subgroup can be worked out
easily as the double coset,
$$\GL_2(K)\backslash \GL_4(k)/P_{(2,2)},$$
can be identified to $\GL_2(K)$-orbits on the set of subspaces $W$
of $V$ of dimension 2 which is easily seen to consist of two orbits,
one represented by a $W$ which is invariant under $K$, and the other
which is not. We omit the details needed here for completing the proof of
theorem 1.2, but remind the reader that we have not been
able to handle representations of $\GL_4(k)$ whose Langlands
parameter is of the form $\sigma_1 \otimes {\rm St}_2$ for a
two dimensional parameter $\sigma_1$ with a nontrivial self-twist;
these are of course the notorious generalized Steinberg
representations.

We end this section by  formulating the following general
conjecture, which is a modified form of a conjecture in
\cite{Prasad1}.

\begin{conj}\label{conj2}
Let $D$ be a quaternion division algebra over a local field $k$, $K$
a quadratic separable algebra over $k$. Let $\pi$ be an irreducible,
admissible, generic representation of $\GL_{2n}(k)$ with central
character $\omega_\pi$. Let $\chi$ be a character of $K^*$, also
considered as a character of $\GL_n(K)$ via the determinant map
$\det: \GL_n(K)  \rightarrow K^*$, such that $\chi^n|_{k^*} =
\omega_{\pi}$. Then the character $\chi$ of $\GL_n(K)$ appears as a
quotient in $\pi$ restricted to $\GL_n(K)$  if and only if

\begin{enumerate}

\item The Langlands parameter of $\pi$ takes values in $\GSp_{2n}({\Bbb C})$
with similitude factor $\chi|_{k^*}$.

\item The epsilon factor
$\epsilon( \pi \otimes {\rm { Ind}_K^k}(\chi^{-1})) = 1. $
\end{enumerate}

Similarly, assuming that $\pi$ can be transferred to a
representation $\pi'$ of $\GL_n(D)$, and that $K$ is a quadratic
field extension of $k$ so that $\GL_n(K)$ embeds into $\GL_n(D)$,
the character $\chi$ of $\GL_{n}(K)$ appears in $\pi'$ restricted to
$\GL_n(K)$ as a quotient if and only if

\begin{enumerate}

\item The Langlands parameter of $\pi$ takes values in $\GSp_{2n}({\Bbb C})$
with  similitude factor $\chi|_{k^*}$.

\item The epsilon factor
$\epsilon( \pi \otimes {\rm { Ind}_K^k}(\chi^{-1})) = -1. $
\end{enumerate}

\end{conj}

\begin{rem} Multiplicity 1 of the trivial character of $\GL_n(K)$ inside
an irreducible admissible representation of $\GL_{2n}(k)$ was proved
by J. Guo in \cite{Guo}, but the multiplicity 1 of more general
characters of $\GL_n(K)$  seems not to have been addressed in the
literature.
\end{rem}

\begin{rem} There is a very related branching law (in that it also
involves root numbers of a symplectic representation which is a tensor
product of a two dimensional orthogonal representation with a
symplectic representation) which should be of considerable interest.
It is to describe those characters of $\U(n)$ factoring
through ${\rm SU(n)}$ which appear
as a quotient in a representation of $\SO(2n+1)$ where $\U(n)
\hookrightarrow \SO(2n+1)$, and the groups are defined over any local field. 
We note that the dimension of
$\SO(2n+1)/\U(n)$ is the same as the dimension of $\SO(2n+1)/U$
where $U$ is the unipotent radical of the Borel subgroup of $\SO(2n+1)$,
so by some heuristic `most' representations of $\SO(2n+1)$ may have a
linear form invariant under $U(n)$.
\end{rem}

\section{Dual pairs involving division algebras: Archimedean Case}
\label{Li}

 Although it is not strictly  necessary to discuss theta
correspondence in the Archimedean case for this work, but since the
Archimedean theta correspondence was at the source of the work in
\cite{Takloo-Bighash} which we will complete in a later section, it
seems appropriate to discuss it, specially to bring out a difference
in the Archimedean correspondence and the non-Archimedean one.

We dual pairs $(U(V),U(W))$ discussed in the last section without
any specific base field continue to hold good in the case when $k =
{\Bbb R}$, and where the quaternion algebra is the Hamiltonian
${\Bbb H}$. We recall that over ${\Bbb H}$ a skew-Hermitian form of
any rank is unique, and can be taken to be
$$\bar{X}_1X_{n} + \bar{X}_2X_{n-2}+ \cdots + \bar{X}_nX_1.$$
The corresponding unitary group is conventionally written as
$O^*(2n)$; it is a form of $SO(2n)$, of real rank $[n/2]$, and is a
connected group.

Hermitian forms over  ${\Bbb H}$ are classified by a signature, and
up to isomorphism can be taken to be
$$\bar{X}_1iX_{1} + \bar{X}_2iX_{2}+ \cdots  \bar{X}_piX_p - \bar{X}_{p+1}
iX_{p+1} - \cdots -\bar{X}_niX_n.$$ The corresponding unitary group is
conventionally written as $Sp(p,n-p)$; it is a form of $Sp(2n)$, of
real rank ${\rm min}(p,n-p)$, and is a connected group.

The theta correspondence between $O^*(2n)$ and $Sp(p,q)$ is
considered in detail in \cite{Li-etal}, and their main theorem can
be succinctly stated as follows:

\begin{thm}Denoting $\hat{G}$ the isomorphism classes of irreducible
Harish-Chandra modules for a real Lie group $G$, the theta
correspondence gives a bijection of sets relating obvious Harish-Chandra
and Langlands parameters,
$$\bigsqcup_{p+q =n, n-1} \hat{{\rm Sp}}(p,q)  \longrightarrow \hat{O}^*(2n).$$
\end{thm}

\begin{rem}For our purposes, an explicit version of this theorem for
$n=2$ will be what would be relevant, in which case $O^*(4)\cong
SU(2) \times SL_2({\Bbb R})$. A discrete series representation of
$SU(2) \times SL_2({\Bbb R})$ is of the form $F_n \boxtimes D_{\pm m}$,
for integers $n>0, m>0$, where $F_n$ is the finite dimensional
representation of $SU(2)$ of dimension $n$, and $D_{\pm (m)}$ is a
discrete series representation with highest (lowest) weight $m+1$,
resp.  $-(m+1)$. Then the theorems of \cite{Li-etal} imply that  the
theta lift of $F_n \boxtimes D_m$ to $Sp(1,1)$ is zero if $n \leq |m|$.
The representations $F_n \boxtimes D_{\pm n}$ arise from theta lift of
$F_n$ from $Sp(1,0)=Sp(0,1) = SU(2)$, and the representations $F_n
\boxtimes D_m$ for $n < |m|$ from the compact groups $Sp(2,0)$ and
$Sp(0,2)$.
\end{rem}

\begin{rem}\label{weird} As $GO^*(4) = [D^* \times GL_2(k)]/k^*$,
there is a curious involution, call it $\iota$, on the discrete
series representations of this group taking $\pi_1 \boxtimes \pi_2$ to
$JL(\pi_2) \boxtimes JL(\pi_1)$ where for a representation $\pi$ of
$D^*$, $JL(\pi)$ denotes the representation of $GL_2(k)$ obtained
from $\pi$ by the Jacquet Langlands correspondence, and vice-versa.
A  consequence of the previous remark is that for discrete series
representations $\pi$ of $O^*(4)$, for $k = {\Bbb R}$, theta lift to
$Sp(1,1)$ is nonzero for exactly one of the representations $\pi$ or
$\iota(\pi)$, but that both of them might be nonzero in the
non-Archimdean case.
\end{rem}

\section{Discrete series over the reals}

It is well-known that automorphic representations associated to
holomorphic Siegel modular forms are not generic; that is, they fail
to have Whittaker models. It is also known that the genericity of
such representations specifically fails at the archimedean place.
For this reason it is desirable to determine when holomorphic
discrete series representations posses Bessel models which seem to
be the next best thing in applications to $L$-functions
\cite{Furusawa, Furusawa-Shalika}. In the proof of our theorem in
the next paragraph we will use some results of Kolk and Varadarajan
\cite{K-Var} on invariant distributions. For the convenience of the
reader, and for ease of reference,  we  include a review of their
results here.

\subsection{A review of some results of Kolk and Varadarajan}
The heart of the matter in this paragraph is Proposition \ref{3637}
followed by Theorem \ref{thm315}. We refer the reader to the
 paper of \cite{K-Var} for historical comments and various
applications of Theorem \ref{thm315}.

\subsubsection{Transverse symbol} Let $X$ be a $C^\infty$ manifold of dimension $n$. If
$E$ is a Fr\'{e}chet space, we let $E'$ be its dual space provided
with strong dual topology. We write ${\mathcal Dist}(X; E) : =
C_c^\infty(X;E)'$ for the space of $E$-distributions on $X$. For $r
\in \Z_{\geq 0}$, an $E$-distribution $T$ is said to be of order
$\leq r$ if for every compact set $K \subset X$ there exist a
constant $C>0$, a finite number of elements say $D_1, \dots, D_m \in
{\mathcal D}^{(r)}(X)$ (differential operators of order $\leq r$),
and a continuous semi-norm $\nu$ on $E$ such that for all $f \in
C_c^\infty(K;E)$
$$
|\langle T, f\rangle | \leq C \sum_{1 \leq j \leq m} \sup_{x \in K}
\nu(D_j f(x)).
$$
We write ${\mathcal Dist}^{(r)}(X; E)$ for the space of all such
distributions. Now let $Y$ be a closed $C^\infty$ submanifold
embedded in $X$ and of dimension $q$. Set $p = n-q$. Let ${\mathcal
Dist}_Y(X;E)$ be the collection of $E$-distributions with support in
$Y$. If $x \in Y$, select an open neighborhood $U$ of $x$ and local
coordinates $(t, u) = (t^1, \dots, t^p, u^1, \dots, u^q)$ on $U$
such that $Y \cap U$ is $\{(t, u) ; t^1 = \dots = t^p =0\}$. We say
$T \in {\mathcal Dist}(X; E)$ has transverse order $\leq r$ at $x
\in Y$ if there exits an open neighborhood $U$ of $x$ in $X$ such
that $\langle T, f \rangle =0$ if $f \in C_c^\infty (X; E)$
satisfies $Vf|_{Y \cap U} =0$ for all $V \in {\mathcal D}^{(r)}(U)$.
We let ${\mathcal Dist}_Y^{(r)}(X; E)$ be the linear space of
such distributions.

We will now define a vector bundle $M^{(r)}$, the $r$th graded
subspace of the transverse jet bundle over $Y$. For any $x \in Y$,
we let ${\mathcal O}_x$ be the algebra of germs of $C^\infty$
functions around $x$, and ${\mathcal D}_x^{(r)}$ the ${\mathcal
O}_x$-module of germs of differential operators of order $\leq r$
around $x$. Let $V_x^{(r)}$ be the ${\mathcal O}_x$-submodule of
${\mathcal D}_x^{(r)}$ generated by products of $ \leq r$ of germs
of vector fields around $x$ for which at least one is tangent to
$Y$. Set ${\mathcal I}_x^{(r)} = {\mathcal D}_x^{(r-1)} +
V_x^{(r)}$. It turns out that ${\mathcal I}_x^{(r)}$ is the stalk of
a subsheaf ${\mathcal I}^{(r)}$ of ${\mathcal D}^{(r)}$. Then we set
${\mathcal M}^{(r)} ={\mathcal D}^{(r)} /{\mathcal I}^{(r)}$ with
stalk at $x$ denoted by ${\mathcal M}_x^{(r)}$. We write $\partial
\mapsto \overline{\partial}$ for the projection ${\mathcal
D}_x^{(r)} \to {\mathcal M}_x^{(r)}$. If we have local coordinates
$(t, u)$ as above, then $\overline{\partial}_t^\alpha$ with
$|\alpha| = r$ forms a free basis for the sections of $M^{(r)}$
around $x$. This shows that ${\mathcal M}^{(r)}$ is in fact a vector
bundle on $Y$. We let ${\mathcal M}^{(r)'}$ be the dual bundle. We
observe that ${\mathcal M}^{(r)} \otimes {\mathcal Dist}(Y; E)
\simeq {\mathcal Dist}(Y;{\mathcal M}^{(r)'} \otimes E)$. We can now
define the transverse symbol of an $E$-distribution supported on $Y$
as an ${\mathcal M}^{(r)'} \otimes E$-distribution living on $Y$.
Let $T \in {\mathcal Dist}_Y^{(r)}(X;E)$. Then for each $x \in Y$
there exist a neighborhood $U$ of $x$ and local coordinates $(t, u)$
as above, and uniquely determined distributions $\tau_\alpha \in
{\mathcal Dist}(Y; E)$ for each $\alpha$ with $|\alpha| \leq r$ such
that on $U$
\begin{equation}\label{schwartz}
T = \sum_{|\alpha| \leq r} (-1)^{|\alpha|} \tau_\alpha
\partial_t^\alpha.
\end{equation}
The transverse symbol of $T$ then is $\sigma(T) =\sum_{|\alpha| = r}
(-1)^{|\alpha|} \partial_t^\alpha \otimes \tau_\alpha$. The symbol
is uniquely determined and is independent of local coordinates. Also
the map $\sigma^{(r)}: {\mathcal Dist}_Y^{(r)}(X;E) \to {\mathcal
Dist}(Y; {\mathcal M}^{(r)'} \otimes E)$ is injective modulo
${\mathcal Dist}_Y^{(r-1)}(X;E)$.

\subsubsection{Invariant distributions} We will retain the notations and
conventions of the previous paragraph. Let $H$ be a Lie group of
$C^\infty$ diffeomorphisms of $X$ which leave $Y$ invariant. Assume
that the action of $H$ on $Y$ is transitive. Let $H' \subset H$ be a
closed subgroup, and suppose we are given a differentiable action
$\beta$ of $H'$ on $E$. Then it makes sense to talk about
$E$-distributions on $X$ invariant under $H'$. In applications, such
as those considered in the next paragraph, it is sometimes desirable
to know when there are no invariant distributions under $H'$. Here
we will concentrate on ${\mathcal Dist}^{(r)}(X;E)^{H'}$ for $r \geq
0$ and will study situations where this space is trivial. First
observation is that if $Z \subset X$ is open and is $H'$ invariant,
then if we set $U = Z \cap Y$, the set $U$ is closed in $Z$ and the
map $\sigma^{(r)}$ defined on ${\mathcal Dist}^{(r)}_U(Z;E)^{H'}$ will
have its image in ${\mathcal Dist}(U; {\mathcal M}^{(r)'} \otimes
E)^{H'}$.

Now let $F$ be any $H$-homogeneous $C^\infty$ vector bundle of
finite rank on $Y$, and let the action of $H$ be $\alpha$. Then we
have a natural action of $H'$ on $F \otimes E$. We first describe
the structure of ${\mathcal Dist}(Y; F \otimes E)$. Fix an arbitrary
point $x \in Y$, and let $W_0 = F_x \otimes E$. Let $W$ be the
trivial bundle over $H$ whose fibers are isomorphic to $W_0$. Let
$H_x$ be the stabilizer of $x$. Suppose $U$ is an $H'$-invariant
open set in $Y$, and set $V= \pi^{-1}(U)$ where $\pi: H \to Y$ is
given by $h\mapsto h.x$. We define a structure of $H'$-module on
$\Gamma_c^\infty(V;W)$ be setting
$$
(h'.s)(h) = (id \otimes \beta)(h')s(h'^{-1} h),
$$
$h' \in H'$, $h \in H$, and $s \in \Gamma_c^\infty(V; W)$. Also
define $\delta_x$ be the $\R_{>0}$ valued homomorphism on $H_x$
given by $|\det(Ad|_{{\mathfrak h}_x})|$. Note that
$\Gamma_c^\infty(V;W)$ has the structure of an $H_x$ module via
$$
(R(\xi)s)(h) = (\alpha \otimes id)(\xi)s(h \xi)
$$
for $\xi \in H_x$. Then we have the following theorem:
\begin{thm}[Theorem 3.2 of \cite{K-Var}]
There exists an injective continuous mapping of $H'$-modules
$$
\sharp: {\mathcal Dist}(U; F \otimes E) \to {\mathcal Dist}(V; W)
$$
linear over $C^\infty(Y)$ satisfying $\delta_{H_x}(\xi)^{-1} R(\xi)
\circ \sharp = \sharp$ for $\xi \in H_x$. Also $supp \, (\sharp
\tau) \subset \pi^{-1}(supp \, \tau)$.
\end{thm}
Now assume that $H'$ is normal in $H$, and set $H'_x = H' \cap H_x$.
Define a homomorphism $\chi_x: H_x' \to \R$ by $\chi_x(\xi) =
\delta_{H'}(\xi) / \delta_{H_x}(\xi)$. Let $\C_x$ and $\C_x'$ be one
dimensional $H'_x$-spaces determined by $\chi_x$ and $\chi_x^{-1}$,
respectively. Also for each $h \in H$, set $\beta^h(h') : = \beta(h
h' h^{-1})$. Now let $V= H' V$ be open in $H$, and $\zeta \in
{\mathcal Dist}(V;W)^{H'}$. Write $V$ as a union of open sets $V_M =
H' M$ with $M$ a $C^\infty$ manifold in $H$ which is an open subset
of a fixed closed submanifold.

\begin{prop}[Lemmas 3.6 and 3.7 of \cite{K-Var}]\label{3637} To each $\zeta \in
{\mathcal Dist}(V_M, W)^{H'}$ there corresponds a unique $\omega \in
{\mathcal Dist}(M; W)$ such that for all $s \in \Gamma_c^\infty(V_M;
W)$ we have
$$
\langle \zeta, s \rangle = \langle \omega_m, \int_{H'} (id \otimes
\beta)(h'^{-1})s(h'm)\, d_l h' \rangle.
$$
In particular, $supp \, \zeta = H'. supp \, \omega$. If $\zeta$
satisfies $\delta_{H_x}(\xi)^{-1} R(\xi) \zeta = \zeta$ for all $\xi
\in H_x'$, then $\omega$ satisfies
$$
\langle \omega, w \rangle = \langle \omega_m , \chi_x(\xi) (\alpha
\otimes \beta^m )(\xi) w(m) \rangle,
$$
for all $\xi \in H_x'$ and $w \in \Gamma_c^\infty(M;W)$.
\end{prop}

Finally we have the following theorem:

\begin{thm}[Theorem 3.15 of \cite{K-Var}]\label{thm315} Assume we are in one of the
following situations:
\begin{enumerate}
\item The representation $\beta$ extends to $H$ (this includes
$H'=H$);
\item $\dim E < \infty$.
\end{enumerate}
Suppose for all $y \in Y, r \geq 0$, we have
\begin{equation}\label{vanishing}
({\mathcal M}_y^{(r)} \otimes E' \otimes \C_y')^{H_y'} = (0).
\end{equation}
Then ${\mathcal Dist}_Y(X;E)^{H'} = (0)$. In fact, it suffices to
assume the validity of \eqref{vanishing} for one element of every
$H'$-orbit in $Y$.
\end{thm}

\subsection{Discrete series for $\GSp(4, \R)$ and inner forms}

First we describe
the discrete series representations of $\GL(2, \R)$. Let $\eta =
|.|^s sgn^\epsilon$, $\epsilon =0, 1$, be any quasi-character of
$\R^\times$. Then for any positive integer $k$, we have the
following exact sequence of representations
$$ 0 \longrightarrow \delta(\eta, k) \longrightarrow \eta |.|^{k/2}sgn^{k+1} \times \eta |.|^{-k/2}
\longrightarrow \zeta(\eta, k) \longrightarrow 0 .$$

The representation $\zeta(\eta, k)$ is finite dimensional of
dimension $k$. The representation $\delta(\eta, k)$ is essentially
square-integrable, and it is discrete series if $\eta$ is unitary.

We now deal with the group $\Sp(4, \R)$. For every pair of integers
$(p, t)$ with $p > t > 0$ one has a collection of four discrete
series representations of $\Sp(4, \R)$ with the same infinitesimal 
character. We will denote these by
$X(p, t), X(p, -t), X(t, -p), X(-t, -p)$. These representations can
be obtained from the Siegel parabolic subgroup with $\GL_2(\R)$ as
the Levi subgroup
$$
\zeta(|.|^{-\frac{p+t}{2}}sgn^p, p-t)\rtimes 1 \twoheadrightarrow
X(p, t) \oplus X(-t, -p)
$$
and
$$
\delta(|.|^{\frac{t-p}{2}}sgn^p, p+t)\rtimes 1 \twoheadrightarrow
X(p, -t) \oplus X(t, -p).
$$
An essential point for us is the fact that the kernel of the
first map is a finite dimensional representation 
of $\Sp_4({\Bbb R})$, and therefore for calculation of 
Bessel models, there is no difference between the 
principal series and the sum $X(p, t) \oplus X(-t, -p)$.
The representations $X(p, -t)$ and $X(t, -p)$ are generic. 
For $\GSp(4,
\R)$, essentially square-integrable representations, that is
those irreducible representations $\Pi$ whose restriction to $\Sp(4,
\R)$ contains a discrete series representation, will be called a 
discrete series 
representation of $\GSp_4({\Bbb R})$ (thus without
requiring the central character to be unitary).

The group $\GSp_4({\Bbb R})$ contains ${\Bbb R}^*\cdot Sp_4({\Bbb R})$ 
as a subgroup of index 2, and every discrete series representation
of $\GSp_4({\Bbb R})$ is obtained by inducing 
a discrete series representation of ${\Bbb R}^* \cdot Sp_4({\Bbb R})$ 
which thus can be parametrized as $X(p,t;\xi)$ with $\xi$ a character
of ${\Bbb R}^*$ such that $\xi|_{\pm 1}$ is the central
character of the representation $X(p,t)$ of $Sp_4({\Bbb R})$. The action
of $\GSp_4({\Bbb R})$ 
on $\Sp_4({\Bbb R})$ interchanges $X(p,t)$ with
$X(-t,-p)$, and $X(p,-t)$ with $X(t,-p)$.   

Given $(p,t)$ with $p>t>0$, and a character $\xi: {\Bbb R}^*\rightarrow
{\Bbb C}^*$, let $\Pi_1$ be the generic representation of
$\GSp_4({\Bbb R})$ with central character $\xi$, and let $\Pi_2$
be the other discrete series representation of $\GSp_4({\Bbb R})$ 
with the same infinitesimal character. Let $\Pi_3$ be the unique
discrete series representation of $\GSp_{\Bbb H}(4)$ with the 
same infinitesimal and central character.

\subsection{The result}

For a given representation $\pi$, the Bessel functional is
 a continuous linear functional on the space of smooth vectors
$V_\pi^\infty$ in $V_\pi$ which comes equipped with its Fr\'{e}chet
topology satisfying appropriate invariance equations with respect to
the Bessel subgroup. Explicitly, let $\chi$ be a character of
$\C^\times$ given by $\chi(re^{i\theta}) = \chi_1(r)e^{in\theta}$,
for some quasi-character $\chi_1$ of $\R_+^\times$. Given $n$ and
$\chi$ as above, we set $n(\chi) = n$. We identify $\C^\times$ with
a subgroup of $\GL_2(\R)$, and $D^\times$, by sending $z= a+ib
\mapsto t(z) : =
\begin{pmatrix} a& b \\ -b & a \end{pmatrix}$. Define a subgroup $R$
of $\GSp_4(\R)$ by setting
\begin{equation*}
R = \left\{ b(z; r, s, t): = \begin{pmatrix} t(z) \\ & t(z)
\end{pmatrix}
\begin{pmatrix} 1& & s & r \\ & 1 & r & t \\ &&1 \\ &&&1 \end{pmatrix};
r, s, t \in \R, z \in \C^\times \right\}.
\end{equation*}

We now
define a character $\chi_R$ of $R$ by setting
$$
\chi_R(b(z; r, s, t)) = \chi(z)e^{2 \pi i (s+t)}.
$$
There is a closely related subgroup $R_D$  of $\GSp_D(4)$. 
One defines a similar character of $R_D$, again denoted by
$\chi_R$. We say a continuous functional $\lambda$ on $V_\pi^\infty$
is a $\chi$-Bessel functional if it satisfies
$$
\lambda(\pi(r)v) = \chi_R(r)\lambda(v),
$$
for all $v \in V_\pi^\infty$ and $r \in R$. We define $\chi$-Bessel
functionals for representations of $\GSp_D(4)$ similarly.

In the following theorem we are interested in the existence of
Bessel functionals for the representations $\Pi_i$.

\begin{thm}\label{bessel-discrete}
Let $\chi$ be a character of $\C^\times$ as above, and let $\{\Pi_1,
\Pi_2, \Pi_3 \}$ be a Vogan packet associated with a pair of
integers $(p,t)$ subject to $p > t > 0$, in such a way that
$\chi|_{\R^\times}$ is the same as the central character of $\Pi_1$.
Then exactly one of the representations $\Pi_i$, $1 \leq i \leq 3$,
has a $\chi$-Bessel model. More precisely
\begin{enumerate}
\item $\Pi_1$ has the model if and only if $|n(\chi)| > p + t$;
\item $\Pi_2$ has the model if and only if $|n(\chi)| \leq p-t$; and
\item $\Pi_3$ has the model if and only if $p-t < |n(\chi)| \leq
p+t$.
\end{enumerate}
In each case the space of the functionals is one dimensional.
\end{thm}

A few remarks are in order. The theorem is of course the
Gross-Prasad conjecture for discrete series representations of 
$\GSp_4({\Bbb R})$ though we will not check the condition on local 
epsilon factors. We
observe that the parity of $n(\chi)$ is opposite that of $p+t$ and
$p-t$. Theorem \ref{bessel-discrete} completes the work
\cite{Takloo-Bighash}. 
We recall that  \cite{Takloo-Bighash} proved the existence  of 
Bessel functionals using global theta correspondence. It may be desirable to
find a direct local proof of the existence theorem as in Wallach's
paper \cite{Wallach3}. The paper \cite{Wallach3} applies directly
to  the case
of holomorphic representations of $\GSp_4({\Bbb R})$, or discrete
series representations of $\GSp_D(4)$, which are induced from finite 
dimensional representations of Levi subgroups (corresponding to the Siegel
parabolic), and yields the desired result in these cases. (As already
noted for $\GSp_4({\Bbb R})$, these discrete series representations
can be obtained as quotients of principal series representations for which
the kernel is a finite dimensional representation, so does not matter for
calculation of the Bessel models.)

\begin{proof}[{\bf Proof of Theorem \ref{bessel-discrete}}]
In the rank one situation the results follow from Wallach's results.
The same is true of holomorphic discrete series representations as
such representations are quotients of inductions from finite
dimensional representations. It remains to deal with the generic
discrete series representations. The sufficiency of the condition is
one the main results of \cite{Takloo-Bighash}. We just need to
verify the necessity of the conditions. Here as in \cite{Wallach},
we will use the results of Kolk and Varadarajan. Our result will
follow from the following claim:
\begin{claim} Suppose the $\Pi$ is a
quotient of the ${\rm Ind}(\pi | P, G)$ 
with $\pi$ an irreducible
representation of $\GL(2)$. Then if $\Pi$ has a $\chi$-Bessel
functional,  there is a continuous functional
$\lambda$ on $V_\pi^\infty$ satisfying $\lambda(\pi(t(z))v = \chi(z)
\lambda(v)$ for all $v \in V_\pi^\infty$, $z \in \C^\times$; such 
linear forms will be called Waldspurger functional.
\end{claim}

Suppose
$\pi$ acts on a space $V_\pi$. 
By the definition of an induced representation, a Bessel functional 
on  ${\rm Ind}(\pi | P, G)$ defines a distribution $T$ on the space
of $V_\pi$ valued Schwartz functions on $G = \GSp_4(\R)$ satisfying
\begin{enumerate}
\item $T(L_p F) = T(\pi(p)^{-1}F)$, for $p \in P$
\item $T(R_r F) = \theta(n)\chi_n(t) T(F)$, for $r=nt \in R$.
\end{enumerate}
Consider the Bruhat decomposition of $G$ as $P \times P$ double
cosets written as
$$
G= P \cup Pw_1 P \cup P w_2 P,
$$
with $P w_2 P$ the unique open cell. The element $w_1$ can be
represented by the following matrix
$$
w_1 = \begin{pmatrix} 1 \\ & & & 1 \\ && 1 \\ & -1 \end{pmatrix}.
$$

\

The idea here is to show first that $T$ restricted to the open cell
is non-zero.

\

\noindent {\em Step 1.} First step is to show that $T$ restricted to
the open set $Pw_1 P \cup P w_2 P$ is non-zero. If it were zero,
then $T$ would be supported on $P$. We will show that there are no
distributions supported on $P$ satisfying the invariance properties.
In fact we do not need the entire group $P \times R$; $P \times N$
is sufficient. We note that $P \times N$ acts transitively on $P$,
and we can use the Vanishing Theorem \ref{thm315}. In the notation
of that theorem set $x = e$. Hence $H_e' = \{(n, n) |n \in N\}$. A
nice transversal Lie algebra for $P$ at $e$ can be taken to be $L_0
= Lie(\overline{N})$. It is then not hard to see that $H_e'$ acts
trivially on $M_e^{(r)'}$. Also the character $\chi_e$ is trivial.
It is then clear that as the character $\theta$ is non-trivial, we
get
$$
({\mathcal M}_e^{(r)'} \otimes V_\pi' \otimes \C)^{H_e'} = (0).
$$
The vanishing theorem now gives us the result.

\

\noindent {\em Step 2.} We now consider the restriction of $T$ to
the open set $Pw_1 P \cup P w_2 P$. We would like to show that the
restriction of $T$ to $Pw_2 P$ is non-zero. We show that there are
no distributions supported on $Pw_1P$ satisfying the invariance
properties. Here too we just need to use $P \times N$, but
unfortunately the Vanishing Theorem \ref{thm315} does not apply
directly as the action of $P \times N$ is not transitive on $Pw_1P$;
$\pi$ is not finite dimensional; and the action of $P \times N$ does
not extend to $P \times P$. We can however use Proposition
\ref{3637}. Here $H = P \times P$, and $H' = P \times N$. We let
$x=w_1$. A nice transversal Lie algebra for $Pw_1 P = P w_1 ( P \cap
w_1^{-1} \overline{N} w_1)$ at $w_1$ will be $L_0 = Lie(\overline{N}
\cap w_1^{-1} \overline{N} w)$. Then $H_{w_1}= \{ (p, w_1^{-1} p
w_1) | p \in P \cap w_1 P w_1^{-1} \}$. Then action of an element of
$H_{w_1}$, say $(p, w_1^{-1} p w_1)$, on the Lie algebra of $G$ is
given by $Ad(w_1^{-1} p w_1)$. Then $H_{w_1}' = \{ (p, w_1^{-1} p
w_1) | p \in P \cap w_1 N w_1^{-1} \}$. In coordinates
$$
H_{w_1}' = \left\{\left( \begin{pmatrix} 1& r & s& 0 \\ & 1 & 0 & 0 \\
&& 1 \\ && - r & 1 \end{pmatrix},
\begin{pmatrix} 1& & s & -r \\ & 1 & -r & 0 \\ &&1 \\ &&&1
\end{pmatrix} \right); r, s \in \R  \right\}
$$
Again $H_{w_1}'$ acts trivially on ${\mathcal M}_{w_1}^{(r)'}$. We
will show that the only distribution $\omega$ that satisfies
Proposition \ref{3637} is the zero distribution. In the notation of
Proposition \ref{3637}, set
$$
L_\xi(m) = id - (\alpha \otimes \beta^m)(\xi)
$$
for $\xi \in H_{w_1}'$. We will think of $\xi$ as a pair of
unipotent matrices as above. Let $L$ be the Levi factor of $P$. We
claim that for any compact set $K \subset L$, there is $\xi \in
H_{w_1}'$ such that for all $m \in K$, the operator $L_\xi(m)$ is
invertible. This will prove the assertion about $\omega$. Indeed, we
know by Proposition \ref{3637} that $\langle\omega_m,
L_\xi(m)f(m)\rangle = 0$ for all $\xi \in H_x'$. Suppose an
arbitrary $F \in \Gamma_c(M, W)$ is given. Choose $\xi$ so that for
all $m \in \, Supp \, F$, the operator $L_\xi(m)$ is invertible. Set
$f(m) = L_\xi(m)^{-1} F(m)$. Then
\begin{equation*}
\langle\omega_m, F(m)\rangle = \langle\omega, L_\xi(m)f(m)\rangle =
0.
\end{equation*}
Hence we have to prove the claim about the invertibility of
$L_\xi(m)$. Again it is not hard to verify that $\alpha(\xi)$ is
trivial. Hence $L_\xi(m)$ is of the form $id - id \otimes X$ for an
operator $X$. It is easy to see that for such an operator to be
invertible it is sufficient that $X - id$ is invertible. In our
case, $X = \beta^m(\xi)$. We set $r=0$. Then we get
$$
X - id = \left(\psi\left( \tr \begin{pmatrix} s & 0 \\ 0 & 0
\end{pmatrix} \lambda^{-1} g g^t \right) - 1 \right) id
$$
whenever $m = \begin{pmatrix} g \\ & \lambda g^{-t} \end{pmatrix}
\in L$. If $g = \begin{pmatrix} a & b \\ c & d \end{pmatrix}$, then
we get
$$
X - id =\left(\psi\left( s\lambda(a^2 + b^2)\right) - 1 \right) id.
$$
Clearly for each $g$, $a^2 + b^2 \ne 0$. We just need to choose $s$
so that for each each $\lambda, g$, $0 < |s\lambda(a^2 + b^2)| < 1$.
This is possible as the set $K$ is compact.

\

\noindent {\em Step 3.} The above two steps show that the
restriction of $T$ to $P w_2 P = P w_2 N$ is a non-zero
distribution. Note that $P \times N$ acts transitively on $P w_2 N$.
So in the notation of Proposition \ref{3637}, we set $H'=H= P \times
N$, and $M=\{ e \}$.  The space of $V_\pi$-valued functions on the
singleton is the same of $V_\pi$. Hence the collection of
distributions on this space is canonically isomorphic to $V_\pi'$.
Consequently Proposition \ref{3637} implies that there is a
continuous functional $\lambda$ on $V_\pi$ such that
$$
T(f) = \left\langle \lambda, \int_P \int_N \sigma^{-1}(p)
\overline{\theta(n)} f(p w u) \, d_l(p) \, dn \right \rangle.
$$
Note that this makes sense as the inner integral is in fact a vector
in $V_\pi$. In order for $T(f)$ to be invariant under the torus in
$R$, the functional $\lambda$ has to be a $\chi$-Waldspurger model
for $\pi$. This shows that if $\pi$ does not have a
$\chi$-Waldspurger functional, then $\Pi$ does not have a
$\chi$-Bessel functional.

\

\noindent {\em Step 4.} Suppose we have two different Bessel
functionals, and we consider the associated distributions $T_1,
T_2$. By the above step there are two Waldspurger functionals
$\lambda_1, \lambda_2$ corresponding to $T_1, T_2$, respectively. By
the uniqueness of Waldspurger functionals, there is a constant $c$
such that $\lambda_2 = c \lambda_1$. Now we consider the
distribution $T = T_2 - c T_1$. This is a distribution satisfying
the same invariance properties, and further $T|_{Pw_2 P} \equiv 0$.
Consequently, $T=0$.

\

This finishes the proof of the theorem \ref{bessel-discrete}.

\end{proof}

\begin{rem}
This proof, especially the ``inductive" process on the dimension of
the double cosets, is by now fairly standard in harmonic analysis on
Lie groups. Compare Step 1 and 2 of the proof with Proposition 2.10
of \cite{Shalika}. In the rank one case, Shalika had outlined this
argument to one of the authors back in 2002. Also see
\cite{Wallach3, Wallach}.
\end{rem}

\section{The global correspondence for the dual pair $(\GSp, \GO)$}

We now turn to the global setting. Following \cite{Harris-Kudla} we
describe the theta correspondence for the dual pair $(\GSp, \GO)$.
Let $F$ be a number field and let $W, \langle\, . \, \rangle$ (resp.
$V, (\, . \,)$) be a non-degenerate symplectic (resp. orthogonal)
vector space over $F$ with ${\rm dim}_F W = 2n$ (resp. ${\rm dim}_F
V=m$). Let $G=\GSp(W)$ and $H=\GO(V)$. Also let $\mathbb{W}= V
\otimes W$ and $\langle\langle\, . \, \rangle\rangle = (\, . \,)
\otimes \langle \, . \, \rangle$, so that $G$ and $H$ form a dual
reductive pair in the similitude group $\GSp(\mathbb{W})$. More
precisely, we may view $\GSp(W)$ (resp. $\GSp(\mathbb{W})$) as
acting on $W$ (resp. $\mathbb{W}$) on the right. Then, if $h \in
\GO(V)$ and $g \in \GSp(W)$, we define $i(g, h) \in
\GSp(\mathbb{W})$, by $(v \otimes w).i(g, h)= h^{-1}v \otimes w g$.
Note that if $\nu$ denotes the similitude character for 
the various groups involved, then $\nu(i(g,h))=\nu(g)\nu(h)^{-1}$. Let
\begin{equation*}
R=\{ (g, h) \in G\times H\, \vert \, \nu(g) = \nu(h)\,\},
\end{equation*}
so there is a natural homomorphism $i: R \to \Sp(\mathbb{W})$.
Note that if we let $G_1 = \Sp(W)$ and $H_1=\O(V)$, then $G_1
\times H_1 \subset R$.

From now on assume that $m={\rm dim}_F V$ is even, and fix a
non-trivial character $\psi$ of $\A=\A_F$ trivial on
$F$. Let $ W = W_1 \oplus W_1^\vee$ denote a complete polarization
of the symplectic space $W$. Let $\omega=\omega_\psi$ denote the usual action of $G_1(\A)$
on the Schwartz-Bruhat space $\mathcal{S}((V \otimes W_1^\vee)(\A))$
of $(V \otimes W_1^\vee)(\A)$.
For $h\in H(\A)$ and $\varphi \in \mathcal{S}((V \otimes W_1^\vee)(\A))$, let
\begin{equation*}
L(h)\varphi(x)= |\nu(h)|^{-mn/2}\varphi(h^{-1}x).
\end{equation*}
Since $(\det h)^2=\nu(h)^m$, these operators are unitary with respect to the 
natural pre-Hilbert space structure on the Schwartz-Bruhat functions. Note
that the actions of $G_1(\A)$ and $H_1(\A)$ on
$\mathcal{S}((V \otimes W_1^\vee)(\A))$
commute, and
are the usual ones associated to the dual pair $(G_1, H_1)$. On
the other hand, it is not difficult to check the following lemma:
\begin{lem}
For $g_1 \in G_1(\A)$ and $h\in H(\A)$,
\begin{equation*}
L(h^{-1})\omega(g_1)L(h)=\omega( \begin{pmatrix}1 \\&
\nu(h)^{-1}\end{pmatrix} g_1 \begin{pmatrix} 1 \\ & \nu(h)
\end{pmatrix}).
\end{equation*}
\end{lem}

Next, observe that we have an
isomorphism $G_1 \times H \tilde{\to} R$ given by
\begin{equation*}
(g_1, h) \mapsto ( g_1. \begin{pmatrix} 1 \\ & \nu(h)
\end{pmatrix}, h),
\end{equation*}
whose inverse is given by
\begin{equation*}
(g,h)\mapsto (g.\begin{pmatrix} 1 \\ & \nu(h)^{-1}
\end{pmatrix}, h).
\end{equation*}
Thus, we obtain a representation, again denoted by $\omega$, of
the group $R(\A)$ on $\mathcal{S}((V \otimes W_1^\vee)(\A))$, given by
\begin{equation*}
\omega(g, h)\varphi(x) = \omega(g_1).L(h)\varphi(x)
\end{equation*}
where $g=g_1. \begin{pmatrix} 1 \\ & \nu(g)^{-1}
\end{pmatrix}$. Note that the restriction of $\omega$ to the
subgroup $G_1(\A)\times H_1(\A)$ is just the usual action of the
dual pair.

For $(g, h) \in R(\A)$ and $\varphi \in \mathcal{S}((V \otimes W_1^\vee)(\A))$, let
\begin{equation*}
\theta(g, h; \varphi) = \sum_{x \in (V \otimes W_1^\vee)(F)} \omega(g,
h)\varphi(x).
\end{equation*}
It is then well-known
that $\theta(g, h; \varphi)$ is invariant under
$R(F)$. For $\varphi \in \mathcal{S}((V \otimes W_1^\vee)(\A))$ and a
cusp form $f \in \mathcal{A}_0(H)$, consider
the integral
\begin{equation*}
\theta(f; \varphi)(g) = \int_{H_1(F)\backslash H_1(\A)} \theta(g,
h_1 h; \varphi)f(h_1 h)\, dh_1
\end{equation*}
where $h \in H(\A)$ is any element such that $\nu(g) = \nu(h)$ and $dh_1$
is a Haar measure on $H_1(F) \backslash H_1(\A)$.

It is
easy to check that the integral defining $\theta(f; \varphi)$ is
absolutely convergent and is independent of the choice of $h$. One
can also check that $\theta(f; \varphi)$ is left-invariant under
$$\{\gamma \in G(F)\, | \, \nu(\gamma) = \nu(\gamma'), \text{ for
some }\gamma' \in H(F)\}.$$ As far as the central characters are
concerned, it's not hard to see that if the central character of
$f$ is $\chi$, then the central character of $\theta(f; \varphi)$
is $\chi.\chi_V^n$, where $\chi_V(x) = (x, (-1)^{m/2}\det V)$ is
the quadratic character associated to $V$, and therefore for $n$ even,
the central character of $\theta(f; \varphi)$ is $\chi$.

\begin{rem}
One usually defines 
$\theta(f;\varphi)(g)$ 
by integration on 
$H_1(F)\backslash H_1({\Bbb A})$ for $H_1 = O(V)$. However, if $f$ 
belongs to an automorphic representation of $GO(V)({\Bbb A})$
which does not remain irreducible when restricted to $GSO(V)({\Bbb A})$,
then the space of automorphic functions on $GSp(W)$ defined by
\begin{equation*}
\theta^0(f; \varphi)(g) = \int_{H_{1,0}(F)\backslash H_{1,0}(\A)} \theta(g,
h_1 h; \varphi)f(h_1 h)\, dh_1
\end{equation*} 
with $H_{1,0} = SO(V)$, is the same space of functions as those 
obtained as $\theta(f;\varphi)(g)$. We will use this well-known
observation, and use $\theta^0$ instead of $\theta$ in what follows.
\end{rem}

\subsection{Bessel Models} We recall the notion of Bessel model
introduced by Novodvorsky and Piatetski-Shapiro
\cite{Novodvorsky-PS}. For a symmetric matrix $S \in \GL(2,F)$, define a
subgroup $T = T_S$ of $\GL(2)$ by
\begin{equation*}
T = \{ g \in \GL(2)\, \vert \, \,^t g S g = \det g. S \}.
\end{equation*}
We consider $T$ as a subgroup of $\GSp(4)$ via
\begin{equation*}
t \mapsto \begin{pmatrix} t \\ & \det t. \,^t t ^{-1}
\end{pmatrix}.
\end{equation*}

Let us denote by $U$ the subgroup of $\GSp(4)$ defined by
\begin{equation*}
U = \left \{ u(X) = \begin{pmatrix} I_2 & X \\ & I_2 \end{pmatrix} \,
\vert \, X = \,^t X \right \}.
\end{equation*}
Finally, we define a subgroup $R$ of $\GSp(4)$ by $R = TU$.

Let $\psi$ be a non-trivial character of $F \backslash \A$.
Define a character $\psi_S$ on $U(\A)$ by $\psi_S(u(X)) =
\psi(\tr (SX))$ for $X = \,^t X \in \M_2(\A)$; as $S$ will
be fixed throughout,
we abbreviate $\psi_S$ to $\psi$. Let
$\chi$ be a character of $T(F) \backslash T(\A)$. Denote by
$\chi \otimes \psi$ the character of $R(\A)$ defined by
$(\chi \otimes \psi)(tu) = \chi(t) \psi(u)$ for $t \in
T(\A)$ and $u \in U(\A)$.

Let $\pi$ be an automorphic cuspidal representation of $\GSp_4(\A)$
realized on a space $V_\pi$  of automorphic functions. We assume that
\begin{equation}\label{compatible}
\chi \vert_{\A^\times} = \omega_\pi.
\end{equation}
Then for $\varphi \in V_\pi$, we define a function $B(\varphi,g)$ on
$\GSp_4(\A)$ by
\begin{equation}\label{Bessel}
B(\varphi,g) = \int_{Z_{\A} R_F \backslash R_{\A}} (\chi \otimes
\psi)(r)^{-1}. \varphi(rg) \, dr.
\end{equation}
We say that $\pi$ has a global Bessel model of type $(S, \chi,
\psi)$ if for some $\varphi \in V_\pi$, the function
$B(\varphi,g)$ is non-zero. In this case, the $\C$-vector space of
functions on $\GSp_4(\A)$ spanned by $\{ B(\varphi,g) \, \vert \,
\varphi \in V_\pi \}$ is called the space of the global Bessel model
of $\pi$. We abbreviate $B(\varphi,e) $ to be $B(\varphi)$.

Let $\mu: W_1 \rightarrow V$ 
be a homomorphism of vector spaces
 such that the quadratic form
on $V$ restricted to $W_1$ via $\mu$ is the quadratic form on $W_1$
with respect to which the Fourier coefficients is being calculated
on $\GSp(W)$, i.e., the symmetric matrix $S$ in the notation above, but now we 
prefer to do things in a co-ordinate free way. Let $\GO(W_1)^+$ 
be the subgroup of $\GO(W_1)$ consisting
of those elements for which the similitude factor is the similitude
factor of an element of $\GO(V)$. (It is understood that the
quadratic form on $W_1$  arises from a $\mu: W_1 \rightarrow V$ which is 
 fixed.)  In our applications, $\GO^+(W_1)=
\GO(W_1)$.

A map $\mu: W_1 \rightarrow V$ will be identified to a ($F$-valued) 
point of $V
\otimes W_1^\vee$, also denoted by $\mu$, and therefore for a
function $f \in {\mathcal S}((V \otimes W_1^\vee) ({\Bbb A}))$, it
makes sense to consider $f(\mu)$, as well as $L(h)f(\mu)$ for any $h
\in [\GO(V)\times \GL(W_1)]({\Bbb A})$. Let $\O(W_1^\perp)$ be the
subgroup of $\O(V)$ acting trivially on $\mu: W_1 \rightarrow V$. 
It is a  standard calculation that in the
summation defining the theta function, 
$\theta(\varphi) = \sum_{ \mu: W_1\rightarrow V} \varphi(\mu)$, only
those $\mu$'s contribute to the Fourier coefficient we are looking
at for which the quadratic form on $V$ restricts to the desired
quadratic form on $W_1$. Since such embeddings
$\mu: W_1\rightarrow V$ are conjugate under $\SO(V)$ with stabilizer
$\SO(W_1^\perp)$,
  for an automorphic form $f$ on $\GSO(V)
({\Bbb A})$, $\varphi \in {\mathcal S}((V \otimes W_1^\vee) ({\Bbb A}))$,
and $\chi$ an automorphic form on $\GSO(W_1)({\Bbb A})$
$$B_{\chi, \mu}(\theta^0(f;\varphi))
= \int_{\SO(W^\perp_1)(\A) \backslash \SO_V(\A)} \Lambda_\mu(f,
\chi)(h) L(h ) \varphi({\mu}) \, dh,$$ where $h \in \SO(V)({\Bbb A})$, and
\begin{eqnarray*}
\Lambda_\mu(f, \chi)(h) & =&  \int_{\A^\times \GSO(W_1)
\backslash \GSO(W_1)(\A)} \left [\int_{\SO(W^\perp_1) \backslash
\SO(W^\perp_1)(\A)} f(\delta h(g) h)d \delta \right ] \chi(g)dg \\
& = &  
\int_{\A^\times \G[\SO(W_1^\perp) \times \SO(W_1)](F)
\backslash \G[\SO(W_1^\perp) \times \SO(W_1)](\A)} f(\delta h(g)
h) \chi(g) \, d\delta \, dg,
\end{eqnarray*}
where $h(g) \in \GSO(V)(\A)$ has similitude factor $\nu(g)$,
 preserves the embedding $\mu: W_1 \rightarrow V$, and acts as $g$ 
on $W_1$; we have 
$(\delta,g) \in G[SO(W_1^\perp) \times SO(W_1)] \subset 
GSO(W_1^\perp) \times GSO(W_1)$. For sake of explicitness, we record
the following simple lemma needed for the last equality above.

\begin{lem}
Let $G$ be an algebraic group over a number field $F$, and $N$ 
a normal subgroup, with $H = N\backslash G$. Then for appropriate choice of
Haar measures, the following holds for appropriate choice of
functions $f$ on $G(F)\backslash G({\Bbb A})$
$$\int_{H(F)\backslash H({\Bbb A})} 
\int_{N(F)\backslash N({\Bbb A})} f(ng)dn d\bar{g} = 
\int_{G(F)\backslash G({\Bbb A})} f(g) dg.$$
\end{lem}  

\

The following theorem is now immediate by standard arguments:
\begin{thm}
Let $V$ be an even dimensional non-degenerate quadratic space over a
global field $F$, and $\mu: W_1 \rightarrow V$ a linear map giving
rise to a non-degenerate quadratic form on $W_1$ making it possible
to speak of $\mu$-th Fourier coefficient of an automorphic form on
$\GSp(W_1 \oplus W_1^\vee)$, and hence given an automorphic form
$\chi$ on $\GO(W_1)(\A)$, one can define 
the Bessel coefficient $B_{\chi,\mu}(F)$
for an automorphic form $F$ on $\GSp(W_1 \oplus W_1^\vee)$. Given an
automorphic cuspidal representation $\rho$ of $\GO(V)$, there is a
$f \in \rho$ and $\varphi \in {\mathcal S}((V \otimes
W_1^\vee)(\A))$ such that $B_{\chi, \mu}(\theta(f;\varphi)) \ne 0$
if and only if there is $f \in \rho$ such that $\Lambda_\mu(f,
\chi) \ne 0$.
\end{thm}

Just as in the local case, 
the following 
diagram allows one to identify $(E^* \times E^*)/\Delta F^*$ 
inside $(D^* \times D^*)/\Delta F^*$ as the subgroup 
$G[SO(E) \times SO(E)]$ inside
$GSO(D)=GSO(E \oplus E)$

$$
\xymatrix{ &[E ^*\times E^*]/ \Delta(F^*) \ar@{->}[dl]_{\cong}
\ar@{->}[dr]^{} \\
\G[\SO(E)\times \SO(E)] & & ( D^* \times D^*)/(\Delta F^*)}
$$

Therefore the
integral 
$$\int_{\A^\times \G[\SO(W_1^\perp) \times \SO(W_1)](F)
\backslash \G[\SO(W_1^\perp) \times \SO(W_1)](\A)} f(\delta h(g)
) \chi(g) \, d\delta \, dg,$$
becomes a product of two toral 
integrals on  $E^*{\Bbb A}_F^*\backslash {\Bbb A}_E^*$ on which the
theorem of Waldspurger applies, yielding Theorem \ref{thm1.5} of the
introduction. In the case where the dual pair involves division
algebras one can prove a similar theorem. The proof carries over in
an essentially verbatim manner.

\section{Global implies local}  Suppose we
want to show that a given representation $\pi_w$ of $\GSp_4(k)$
has a specific Bessel model $\lambda_w$. One way to show is to 
prove
that there is an automorphic cuspidal representation $\Pi =
\otimes_v \Pi_v$ of $\GSp_4(\A)$ with a non-zero global Bessel
functional $\Lambda = \otimes_v \Lambda_v$ in such a way that $\Pi_w
= \pi_w$ and $\Lambda_w = \lambda_w$.  As not all local 
representations $\pi_w$ arise as the local component 
of an automorphic representation, this is not always
possible but in some situations works quite well. For example if $w$
is a real place, and $\pi_w$ is a discrete series representation,
then it is the image of local theta correspondence from a discrete
series representation of an orthogonal group $\GO(4)$. Then since
discrete series representations of $\GO(4)$ can be globalized to
automorphic cuspidal representations, we can use global theta
correspondence to construct candidates for $\Pi$. Then one uses
results on non-vanishing of $L$-functions of $\GL(2)$
representations to show that $\Pi$ has global Bessel models with
desired local properties. See \cite{Takloo-Bighash} for details. Let
us explain how this procedure works in the non-archimedean
situation, assuming that the theorem about Bessel models has been proved
in the Archimedean case, as well as for principal series
representations. For this we follow the method
of \cite{P} and \cite{P-S}
to deduce the local theorem from the global theorem. Thus we give
ourselves a local field $k$, a number field $F$, a non-archimedean
place $v$ of $F$ such that $F_v \cong k$, a quadratic extension $E$
of $F$ such that  $k \otimes_FE =K$. 
Let $\Pi_v$ be a cuspidal
representation on $\GSp_4(k)$, and $\chi_v$ a character of $K^*$
considered as a character of $R(k) = K^* N(k)$ extending a 
character $\psi_v$ on $N$, such that ${\rm
Hom}_R[\pi_v,\chi_v] \not = 0$. Assume that $\Pi_v$ comes as a theta
lift of a representation $\tau_v$ of $\GSO_4(k)$. We globalize
$\tau_v$ to an automorphic representation $\tau$ on $\GSO_4({\Bbb
A})$ which is unramified at all finite places outside $v$. This is
standard. Then we construct the automorphic form $\Pi$ on
$\GSp_4(k)$ to be the theta lift of $\tau$.

By \cite{P, P-S}, there exists a Gr\"osssencharacter $\chi$ on
${\Bbb A}_E^*$ with local component $\chi_v$ at a place of $E$ above
$v$ of $F$, such that the map taking $f \in \Pi$ to
$$\int_{R(F){\Bbb A}_F^*\backslash R({\Bbb A}_F)} f(g) \chi^{-1}(g) dg ,$$
is not identically zero. By our global theorem, the period
integrals,
$$\int_{E^*{\Bbb A}_F^*\backslash 
{\Bbb A}_E^*} f_1(g) \chi^{-1}(g) dg ,$$
and
$$\int_{E^*{\Bbb A}_F^*\backslash {\Bbb A}_E^*} f_2(g) \chi^{-1}(g) dg $$
on $\Pi_1$ and $\Pi_2$ respectively are not identically zero. This
implies by Waldspurger's theorem that
$$ L(\frac{1}{2}, BC_E(\Pi_1) \otimes \chi^{-1}) \not = 0,$$
and
$$ L(\frac{1}{2}, BC_E(\Pi_2) \otimes \chi^{-1}) \not = 0.$$

In particular the global root number of these $L$-functions is 1:

$$ 1 = \epsilon(\frac{1}{2}, BC_E(\Pi_1) \otimes \chi^{-1})
= \prod_w \epsilon_w ,$$ and
$$ 1 = \epsilon(\frac{1}{2}, BC_E(\Pi_2) \otimes \chi^{-1})  = \prod_w
\epsilon'_w,$$ which as all the other epsilon factors away from the
chosen one at $v$ are equal to 1, we find that

$$\epsilon_v( BC_E(\Pi_1) \otimes \chi^{-1}) = 1,$$ and
  $$\epsilon_v( BC_E(\Pi_2) \otimes \chi^{-1}) = 1,$$
proving  the local result as a consequence of the corresponding
global theorem. This was a case where both local and global results
could be proved by means of theta correspondence. However, the method
acquires teeth in contexts where global results are available without
local facts, such as in residue characteristic 2.

\section{B\"ocherer, global Gross-Prasad, and the Ichino-Ikeda  conjecture}

We begin with  the following
conjecture of B\"ocherer:
\begin{conj} Let $\Phi$ be a holomorphic cuspidal Siegel eigenform
of degree two and weight $k$ with respect to $\Sp_4(\Z)$. Let
\begin{equation*}
\Phi(Z) = \sum_{T >0} a(T, \Phi) \exp(2 \pi \sqrt{-1} tr (T Z))
\end{equation*}
be its Fourier expansion. For a fundamental discriminant $-D$,
i.e. a discriminant of an imaginary quadratic field
$\Q(\sqrt{-D})$, let
\begin{equation*}
B_D(\Phi) = \sum_{\{T \, \vert \, \det(T) = {D \over 4}\} / \sim }
{a (T, \Phi) \over \epsilon(T)},
\end{equation*}
where $\sim$ denotes the equivalence relation defined by $T_1 \sim
T_2$ when $T_1 = ^t \gamma T_2 \gamma$ for some $\gamma \in
\SL_2(\Z)$ and $\epsilon(T) = \# \{ \gamma \in \SL_2(\Z) \, \vert \,
^t \gamma T \gamma = T \}$.

Then there exists a constant $C_\Phi$ which depends only on $\Phi$
such that
\begin{equation*}
L({1\over 2}, \pi_\Phi \otimes \chi_D) = C_\Phi. D^{-k+1}. \vert
B_D(\Phi)\vert^2,
\end{equation*}
where $\pi_\Phi$ is the automorphic representation of
$\GSp_4(\A_\Q)$ associated with $\Phi$, $\chi_D$ is the quadratic
character of $\A_\Q^\times$ associated with $\Q(\sqrt{-D})$ and the
left hand side denotes the central critical value of the quadratic
twist by $\chi_D$ of the degree four Spinor L-function for
$\pi_\Phi$.
\end{conj}

This is a natural generalization of a well-known theorem of
Waldspurger \cite{Waldspurger} for modular form on the upper half
plane to the case of the Siegel cusp forms. B\"ocherer proved this
assertion in the cases of the Klingen Eisenstein series and the
Saito-Kurokawa lifting in \cite{Bocherer}. Later he and
Schulze-Pillot proved this in the case of the Yoshida lifting in
\cite{BS}. More recently Furusawa and Shalika have started
investigating this conjecture from a different angle
\cite{Furusawa-Shalika}. Their approach to the problem is to
generalize Jacquet's {\em relative trace formula} for $\GL(2)$ to
$\GSp(4)$. Jacquet had used his $\GL(2)$ relative trace formula to
give another proof for the theorem of Waldspurger mentioned above.

Here is the global form of the conjecture of Gross and Prasad in this
context.
\begin{conj} For a globally generic cuspidal automorphic
representation $\Pi$ of $\GSp_4(\A_F)$, $E$ a quadratic extension of
the global field $F$, and $\chi$ a Gr\"ossencharacter of
$\A_E^\times$ such that $\chi$ restricted to ${\Bbb A}_F^*$ is the
central character of $\Pi$, we have:
\begin{equation}\label{L}
L({1\over 2},  \Pi \otimes {\rm ind}_E^F (\chi^{-1})) \ne 0
\end{equation}
if and only if there exists a triple $(D, \Pi_D, \Psi_D)$, where $D$
is a central simple quaternion algebra over $F$ containing $E$,
$\Pi_D$ is a cuspidal automorphic representation of $G_D$, an inner
form of $\GSp_4(F)$ defined using  $D$, which corresponds to $\Pi$
in the functorial sense, and $\Psi_D$ is a cusp form in the space of
$\Pi_D$ such that
\begin{equation}\label{B}
\int_{\A_F^\times R_D(F) \backslash R_D(\A_F)} \Psi_D(r) \tau^{-1}(r)
\, dr \ne 0.
\end{equation}
Here $R_D$ denotes the Bessel subgroup of $G_D$ and $\tau$ is the
character of $R_D(\A_F) = {\A_E^* \cdot N({\A_F})}$ which is $\chi$
on ${\Bbb A}_E^*$, and $\psi$ on $N({\Bbb A}_F)$.

\end{conj}

\begin{rem} The global Gross-Prasad conjecture
for Bessel models for generic representations of $PGSp_4$ has been
proven by Ginzburg, Jiang and Rallis in \cite{ginzburg}.
\end{rem}

We now observe that B\"ocherer's conjecture is a refinement of the
global Gross-Prasad conjecture in our context. In order to see this,
we recall that there is a a result of Sugano according to which for
the torus $T_D$ defined by ${\Bbb Q}(\sqrt{-D})$, the period
integral
$$\int_{\A_{\Bbb Q}^\times R_D({\Bbb Q}) \backslash R_D(\A_{\Bbb Q})}
\Psi_D(r) \tau^{-1}(r) dr$$ is essentially $B_D(\Phi)$ of B\"oecherer's
conjecture. For a modern formulation of Sugano's result, see
\cite{Furusawa}.

\subsection{Ichino-Ikeda}
Generalizing Waldspurger, Ichino-Ikeda \cite{Ichino-Ikeda} 
have made a very precise 
conjecture about period integrals thus refining the Gross-Prasad
conjecture. The precise conjecture of Ichino-Ikeda is remarkable 
for its elegance and simplicity. In essence, what it says is that
when there are multiplicity one theorems, any two invariant linear
forms must be scalar multiples of each other. One of the invariant
linear forms that one takes is the period integral, and the other
is defined as a product of linear forms, $\ell = \otimes \ell_v$,
on $\pi = \otimes \pi_v$
through matrix coefficients of representations $\pi_v$ involved. 
Since most of these constructions are unique only up to isomorphism
(a scaling), one needs subtle care to come up with an expression
which is independent of all choices, and then one can meaningfully 
compare the two linear forms, and the suggestion of Ichino-Ikeda is that
with this care taken, the scaling is essentially the $L$-function
that appears in the Gross-Prasad conjecture besides some simple
factors. We make their suggestion precise in our context.

Let $\pi \cong \otimes \pi_v$ be a cuspidal automorphic representation
of $\G({\Bbb A})= \GSp_4({\Bbb A})$. Assume that $\pi$ is a unitary 
representation on which the hermitian form comes from a fixed realization 
of $\pi$ as a space of automorphic forms on $\G({\Bbb A})$. Assume that
$\pi_v$ also are unitary, and that the unitary structure on 
$\pi_v$ yields a unitary structure on $\pi = \otimes \pi_v$ which is the
same as the one coming from the automorphic realization of $\pi$.

Corresponding to an element $\varphi_v \in \pi_v$, 
define a matrix coefficient of $\pi_v$ by
$$\Phi_{\varphi_v,\varphi_v}(g) = <g\varphi_v,\varphi_v>.$$

Let $R$ be the Bessel subgroup of $G$ as in the previous sections, 
and let $\chi = \prod_v\chi_v$ be a character on $R({\Bbb A})$ 
as before. Define {\it local period integrals} to be,
$$I_v(\varphi_v,\chi_v) 
= \int_{k^*_v\backslash R(k_v)} \Phi_{
\varphi_v,\varphi_v}(r) \chi_v^{-1}(r) dr_v.$$
This requires fixing Haar measures $dr_v$ on $k^*_v\backslash R(k_v)$. We 
assume that they are arbitrarily fixed so that one can speak of
Haar measure on ${\Bbb A}_F^*\backslash R({\Bbb A}_F)$. Let 
$dr$ be the Tamagawa measure on ${\Bbb A}_F^*\backslash R({\Bbb A}_F)$. 
Define a constant $C_0$ by $dr = C_0 \prod_v{dr_v}$.
The following conjecture, besides possible inaccuracies, 
should be taken to be 
due to Ichino and Ikeda.

\begin{conj} With the notation as above, let    
$\varphi = \otimes \varphi_v \in \otimes \pi_v$ be a cusp form 
on $G({\Bbb A})$. Then
$$\frac{\left | \int_{{\Bbb A}_F^*R(F)\backslash R({\Bbb A}_F)} \varphi(r)
\chi^{-1}(r) dr \right |^2 }{ \int_{{\Bbb A}_F^*\G(F)\backslash 
\G({\Bbb A}_F)} |\varphi(g)|^2 dg  } = C_0 2^\beta \zeta^S(2) \zeta^S(4) P^S(\pi, \chi^{-1},
\frac{1}{2}) \prod_{v \in S}\frac{ 
\left |I_v(\varphi_v,\chi_v) \right |^2}{<\varphi_v,\varphi_v>},$$
where $S$ is a finite set of finite primes, the superscript such as 
in $\zeta^S$ means removal of the Euler factors at places in $S$,
 $dg$ is the Tamagawa measure 
on ${\Bbb A}_F^*G(F)\backslash G({\Bbb A}_F)$, and the constant $C_0$ has been
defined earlier,  $\beta$ is a non-negative integer $\leq 2$,
and $$P^S(\pi, \chi^{-1},s) = \prod_{v \not \in S} 
\frac{L_v(s,\pi \otimes {\rm Ind} \chi^{-1})}{ L_v(s+\frac{1}{2}, \pi, Ad) 
L_v(s+\frac{1}{2}, \chi)}.$$ 
\end{conj}

\begin{rem} One of the crucial inputs for the Ichino-Ikeda formulation
is the convergence of the local period integral, say for unramified
representations whose Satake parameters belong to certain half-space,
and its calculation; in our case,  we should be getting
$$I_v(\varphi_v,\chi_v) = \zeta_v(2)\zeta_v(4)\frac{L_v(s,\pi \otimes {\rm Ind} \chi^{-1})}{ L_v(s+\frac{1}{2}, \pi, Ad) 
L_v(s+\frac{1}{2}, \chi)},$$
but we have neither proved the convergence, nor done this calculation here.

\end{rem}
 
\begin{rem}
To conclude B\"ocherer's conjecture from Ichino-Ikeda, one will need to 
prove that for unramified representations $\pi_v$ of $\GSp_4(k_v)$, with 
$\varphi_v$  a spherical vector, $I_v(\varphi_v,\chi_v)$ has a simple
dependence on the subgroup $K_v^*$, and the character $\chi_v$ on it.
\end{rem}

\end{document}